\documentclass[a4paper,12]{article}
\usepackage{amsmath, amsfonts, amssymb, amsthm, bm, graphics, bbm, color}
\usepackage{graphicx}
\usepackage{subfigure}
\usepackage{mathtools}
\usepackage[table]{xcolor}
\usepackage{url}
\usepackage{mathabx}
\usepackage{cancel}
\usepackage{multicol}
\usepackage{float}
\usepackage{caption}
\usepackage{mathalfa}
\usepackage{hyperref}
\usepackage{afterpage}
\usepackage{multirow}
\DeclareMathAlphabet\mathbfcal{OMS}{cmsy}{b}{n}
\newtheorem{theorem}{Theorem}[section]
\newtheorem{lemma}[theorem]{Lemma}
\newtheorem{corollary}[theorem]{Corollary}

\theoremstyle{definition}

\newtheorem{remark}[theorem]{Remark}

\newcolumntype{P}[1]{>{\centering\arraybackslash}p{#1}}

\makeatletter
\newcommand\makebig[2]{%
  \@xp\newcommand\@xp*\csname#1\endcsname{\bBigg@{#2}}%
  \@xp\newcommand\@xp*\csname#1l\endcsname{\@xp\mathopen\csname#1\endcsname}%
  \@xp\newcommand\@xp*\csname#1r\endcsname{\@xp\mathclose\csname#1\endcsname}%
}
\makeatother

\makebig{biggg} {3.0}
\makebig{Biggg} {3.5}
\makebig{bigggg}{4.0}
\makebig{Bigggg}{4.5}
\makebig{biggggg}{5.0}
\makebig{Biggggg}{5.5}
\makebig{bigggggg}{6.0}
\makebig{Bigggggg}{6.5}

\renewcommand{\vec}[1]{\mbox{$\bm #1$}}

\newcommand{\dif}{\mathrm{d}}
\newcommand{\im}{\mathrm{i}}

\newcommand{\IH}{{\text{IH}}}
\newcommand{\MH}{{\text{MH}}}

\begin{document}
\title{Properties of Generalised Magnetic Polarizability Tensors}

\author{P.D. Ledger$^\dagger$  and W.R.B. Lionheart$^\ddagger$\\
$^\dagger$School of Computing \& Mathematics, Keele University \\
$^\ddagger$Department of Mathematics, The University of Manchester\\
Corresponding author: p.d.ledger@keele.ac.uk}

\maketitle

\section*{Abstract}
We present new alternative complete asymptotic expansions for the time harmonic low--frequency magnetic field perturbation caused by the presence of a conducting permeable object as its size tends to zero for the eddy current approximation of the Maxwell system. Our new alternative formulations enable a natural extension of the well known rank 2 magnetic polarizability tensor (MPT) object characterisation to higher order tensor descriptions by introducing generalised MPTs (GMPTs) using multi-indices. In particular, we identify the magnetostatic contribution, provide new results on the symmetries of GMPTs, derive explicit formulae for the real and imaginary parts of GMPT coefficients and also describe the spectral behaviour of GMPT coefficients. We also introduce the concept of harmonic GMPTs (HGMPTs) that have fewer coefficients than other GMPT descriptions of the same order. We describe the scaling, translation and rotational properties of HGMPTs and describe an approach for obtaining those HGMPT coefficients that are invariant under the action of a symmetry group. Such an approach is one candidate for selecting features in object classification for hidden object identification using HGMPTs.

\noindent{\bf Keywords:} Asymptotic analysis, eddy current, inverse problems, magnetic polarizability tensor, metal detection, spectral problems.

\noindent MSC Classification: 35R30; 35B30

\section{Introduction}

Characterising highly conducting objects from low frequency magnetic field perturbations is important in metal detection, where the goal is to locate and identify concealed inclusions in an otherwise uniform background material. Applications of metal detection include airport, transport hub and event security, the search for artefacts of archaeological significance, the investigation of crime scenes using forensic science, the recycling of metals and in the search for landmines and unexploded ordnance. Being able to better characterise objects offers considerable advantages in reducing the number of false positives in metal detection and, in particular, to accelerate and improve object location and discrimination.

Ammari, {\em et al}~\cite{Ammari2014, Ammari2015} have established the leading order term in an asymptotic expansion of the perturbed magnetic field $({\vec H}_\alpha- {\vec H}_0)({\vec x})$ due to the presence of a highly conducting permeable object as its size, $\alpha$, tends to zero, which describes the metal detection problem. In~\cite{LedgerLionheart2015} we have shown that the leading order term in the expansion  includes a complex symmetric rank 2 magnetic polarizability tensor (MPT), which provides an object description. We have obtained explicit formulae for the MPT coefficients that depend on the object geometry, its size, material properties and frequency of excitation. In a series of works, we have explored the properties of MPTs including providing several different splittings and formulations for obtained MPT coefficients~\cite{LedgerLionheart2016,LedgerLionheart2019} and also investigated the spectral properties of their coefficients~\cite{LedgerLionheart2019}. Together with Wilson, we have also developed efficient computational algorithms for the computation of the MPT coefficients and their spectral signature~\cite{ben2020} that have allowed us to generate a large dictionary of object characterisations~\cite{benobjects} and apply machine learning algorithms to identify hidden objects using classification~\cite{benclass}, which exploit the MPT's spectral signature.

While using an MPT's spectral signature offers considerable benefits to using an object characterisation based on an MPT at a fixed frequency, the object description, in this case, is only through at most $6$ independent complex coefficients as a function of frequency, limiting the amount of information that can be said about hidden object and preventing its material description to be separated from its shape. To improve this, a complete asymptotic expansion of $({\vec H}_\alpha- {\vec H}_0)({\vec x})$ due to the presence of a highly conducting permeable object as $\alpha \to 0$ has been established in~\cite{LedgerLionheart2018g}, extending the expansion obtained by Ammari {\it et al.} in~\cite{Ammari2014, Ammari2015}. This expansion provides improved object characterisation through the introduction of 
generalised magnetic polarizability tensors (GMPTs), which, in their simplest form, agree with the MPT object characterisation. The higher order terms in the expansion play an important role when the background field at the position of the object is non-uniform. This is indeed the case in many practical metal detection scenarios where the fields generated by coils at the position of concealed objects is far from uniform and the GMPTs allow this information to be used in a smart way. Recent work has also shown that GMPT coefficients can also be measured in practice and the measurement coefficients agree with numerical simulations~\cite{toykangmpt}. For a related scalar electrical impedance tomography (EIT) problem,  it is known that the complete set of generalised polarisation/polarizability tensors (GPTs) uniquely determine an object's shape and its admittivity~\cite{ammarikangbook}.

While explicit formulae for GMPT coefficients have been established in~\cite{LedgerLionheart2018g}, the properties of GMPTs and the choice of suitable invariants for object classification remains open. This work addresses the properties of GMPTs and introduces the concept of harmonic GMPTs (HGMPTs) building on the harmonic generalised polarizability tensors, which we have introduced for a simpler scalar EIT problem and the concept of (contracted) generalised polarizability tensors ((C)GPTs)~\cite{ammari3dinvgpt}. These object descriptions have fewer independent coefficients than other GPTs of the same order. We also describe an approach for determining the independent coefficients of HGMPTs that are invariant under the action of a symmetry group, which offers possibilities for object classification using (H)GMPTs.
Specifically, the novelties of this work are:

\begin{enumerate}

\item We derive  complete asymptotic expansions of $({\vec H}_\alpha-{\vec H}_0)({\vec x})$ as $\alpha \to 0$ using both tensorial index and multi-index notation, which lead to improved object characterisations using higher order GMPTs that are a natural extension of the rank 2 MPT description first obtained in~\cite{LedgerLionheart2015}.

\item We provide a splitting of the GMPT obtained in 1., extending the result for MPTs in~\cite{LedgerLionheart2016}, which makes the magnetostatic contribution to the GMPT explicit.

\item We derive symmetry properties of GMPTs, extending the known complex symmetric property of MPTs previously obtained in~\cite{LedgerLionheart2015}.

\item We derive explicit formulae for the real and imaginary parts of GMPTs, extending those known for MPTs previously obtained in~\cite{LedgerLionheart2019}.

\item We derive a result giving insights in to the spectral behaviour of the GMPT coefficients, extending what is known for the spectral behaviour of MPT coefficients in \cite{LedgerLionheart2019}.

\item We derive a new form of GMPTs called harmonic GMPTs (HGMPTs), which have coefficients that are invariant under rotation for objects that are a member of 
a particular symmetry group.

\end{enumerate}

The work is organised as follows: We first fix some notation in Section~\ref{sect:notation}. Then, in Section~\ref{sect:math}, we briefly recall the mathematical model. In Section~\ref{sect:complete}, we present a series of alternative complete asymptotic expansions for  $({\vec H}_\alpha - {\vec H}_0)({\vec x})$ and introduce alternative forms of GMPTs using both tensorial and multi-indices. In Section~\ref{sect:gmptprop}, we explore some properties of GMPTs. This is followed by the introduction of the concept of HGMPTs in Section~\ref{sect:hgmpts}. We finish with some concluding remarks.

\section{Notation} \label{sect:notation}
We denote by $  {\vec e}_k$  the unit basis vector associated with the $k$th coordinate direction in a standard orthonormal coordinate system ${\vec x}=(x_1,x_2,x_3)$ and, hence, the $k$th component of a vector field ${\vec v}$ in this system is ${\vec e}_k \cdot {\vec v} = ( {\vec v})_k=v_k$.  We will often use Einstein index summation notation so that a vector can be described as ${\vec v} = v_k {\vec e}_k$ and a rank 2 tensor using a calligraphic font as  $ {\mathcal M} =  {\mathcal M}_{k j} {\vec e}_k \otimes {\vec e}_j$. We will use a Gothic font for higher order tensors so that  a rank 3 tensor using tensorial indices can be described as $  {\mathfrak D} =  {\mathfrak D}_{ij k} {\vec e}_i \otimes {\vec e}_{j} \otimes {\vec e}_k$.  We will also use Gothic font for higher rank tensors that use both tensorial and multi-indices. For example, when considering expressions of the form $\displaystyle \sum_{ \beta,\gamma, |\beta|=|\gamma|=0} {\vec x}^\beta {\mathfrak D}_{ij} ^{\beta\delta} {\vec y}^\delta {\vec e}_i \otimes {\vec e}_{j} $, involving the coefficients ${\mathfrak D}_{ij} ^{\beta\delta}$ of a  rank $2+|\beta|+|\delta|$ rank tensor, with the subscripts  $i$ and $j$ being tensorial indices and the superscripts $\beta $ and $\gamma$ being multi-indices,  and summation is implied over  $i$ and $j$ and is explicit over $\beta $ and $\gamma$.  Here,  the multi-indices $\beta=(\beta_1,\beta_2,\beta_3)$ and $\delta=(\delta_1,\delta_2,\delta_3)$ have properties $\beta! = \beta_1! \beta_2! \beta_3!$, $|\beta| = \beta_1 + \beta_2 + \beta_3$, ${\vec x}^\beta = x_1^{\beta_1} x_2^{\beta_2} x_3^{\beta_3}$, $\partial_x^\beta( \cdot) =\partial_{x_1}^{\beta_1}\partial_{x_2}^{\beta_2} \partial_{x_3}^{\beta_3}(\cdot)$. Finally, we also use Gothic font for higher rank tensors that use both tensorial indices and additional indices associated with instances $m$ and $n$ of polynomials $P_p^m({\vec x})$, $P_q^n({\vec x})$ of degree $p$ and $q$. 
Hence,  when considering expressions of the form $\displaystyle \sum_{p=0} \sum_{q=0} \sum_{m=-p}^p \sum_{n=-q}^q P_p^m({\vec x}) {\mathfrak D}_{ij} ^{pmqn}P_q^m({\vec y}) {\vec e}_i \otimes {\vec e}_{j}$, involving the coefficients ${\mathfrak D}_{ij} ^{pmqn}$ of a rank $2+p +q$ rank tensor, summation is implied over the tensorial indices  while that over the indices associated with instances of the polynomials is explicit.

\section{Mathematical model} \label{sect:math}
We briefly recall from~\cite{Ammari2014, LedgerLionheart2015} the mathematical model of interest in this work: Our interest lies in the characterisation of a single homogeneous conducting permeable object. Following previous work, we describe a single inclusion by $B_\alpha : = \alpha B + {\vec z}$, which means it can be thought of as unit sized object $B$, scaled by $\alpha$ and translated by ${\vec z}$. We assume the background is non--conducting and non--permeable and introduce the position-dependent conductivity and permeability as 
\begin{align}
\sigma_\alpha = \left \{ \begin{array}{ll} \sigma_* & \text{in $B_\alpha$} , \\
0 &  \text{in $B_\alpha^c ={\mathbb R}^3 \setminus \overline{B_\alpha}$} \end{array} \right . , \qquad
\mu_\alpha  = \left \{ \begin{array}{ll} \mu_* & \text{in $B_\alpha$} \\
\mu_0 &  \text{in $B_\alpha^c $} \end{array} \right . , \nonumber
\end{align}
where $\mu_0 : = 4 \pi \times 10^{-7} \text{H/m}$ is the permeability of free space, $0 < \mu_*< \infty$ and $0 \le \sigma_* < \infty$. In principal, $\mu_*$ and $\sigma_*$ do not need to be homogeneous in $B_\alpha$, and we have previously considered this situation for MPT object characterisations in~\cite{LedgerLionheartamad2019}. In this work, we assume $\mu_*$ and $\sigma_*$ are homogeneous in $B_\alpha$ for simplicity of presentation.  We will also use the position dependent relative permeability $\tilde{\mu_r}: = \mu_\alpha / \mu_0$ with $\tilde{\mu}_r= \mu_r:=\mu_* /\mu_0$ inside $B_\alpha$ and $\tilde{\mu}_r =1 $ in $B_\alpha^c$. For metal detection, the eddy current approximation of Maxwell's equations is appropriate, since $\sigma_*$ is large and the angular frequency $\omega = 2\pi f $ is small (a rigorous justification involves the object topology~\cite{ammaribuffa2000}). In this case, the electric and magnetic interaction fields, ${\vec E}_\alpha$ and ${\vec H}_\alpha$, respectively, satisfy
\begin{align}
\nabla \times {\vec H}_\alpha = \sigma_\alpha {\vec E}_\alpha  + {\vec J}_0, \qquad \nabla \times {\vec E}_\alpha = \im \omega \mu_\alpha {\vec H}_\alpha, \label{eqn:eddy}
\end{align} 
in ${\mathbb R}^3$ and decay as $O(1/|{\vec x}|)$ as $|{\vec x}| \to \infty$. In the above, ${\vec J}_0$ is a solenoidal external current source with support in $B_\alpha^c$. In the absence of an object, the background fields ${\vec E}_0$ and ${\vec H}_0$ satisfy (\ref{eqn:eddy}) with $\alpha =0$.

The task is to find an economical description of  $({\vec H}_\alpha - {\vec H}_0)({\vec x})$ at a position ${\vec x}$  away from $B_\alpha$, which characterises the object's shape and material parameters by a small number of parameters separately from its position ${\vec z}$ for the regime where
\begin{align}
\nu := \omega \sigma_* \mu_0 \alpha^2,  \nonumber
 \end{align}
 is order one, $\mu_r$ is also order one as $\alpha \to 0$.

\section{Complete asymptotic expansion}\label{sect:complete}


In the following, we present several different equivalent complete asymptotic expansions for $({\vec H}_\alpha- {\vec H}_0) ({\vec x})$ as $\alpha \to 0$, which allow us to introduce different object chracterisations.

\subsection{Original form using tensor indices}
For comparison with subsequent sections, we first recall the complete asymptotic expansions for $({\vec H}_\alpha- {\vec H}_0) ({\vec x})$ as $\alpha \to 0$ previously derived in~\cite{LedgerLionheart2018g}.

\begin{theorem}[Ledger and Lionheart~\cite{LedgerLionheart2018g}]~\label{thm:main}
For any $M>0$, the magnetic field perturbation in the presence of a small conducting object $B_\alpha = \alpha B +{\vec z}$ for the eddy current model when $\nu $ and $\mu_r$ are order one and ${\vec x}$ is away from the location ${\vec z}$ of the inclusion is completely described by the asymptotic formula
\begin{align}
({\vec H}_\alpha - {\vec H}_0) ({\vec x} )_i =&  \sum_{m=0}^{M-1} \sum_{p=0}^{M-1-m} 
( {\vec D}_{ x}^{2+m} G( {\vec x} , {\vec z} ))_{i,{K}(m+1)} {\mathfrak M}_{{K}(m+1) {J} (p+1)} ({\vec D}_z^p ({\vec H}_0  ( {\vec z} )))_{{J}(p+1)} +\nonumber \\
& ({\vec R}({\vec x}))_i, \label{eqn:mainresult} \\
{J}(p+1):=  &[j ,J(p)] = [j ,j_1,j_2,\cdots, j_p] , \nonumber \\
 {K}(m+1) :=&[k ,K(m)] = [k ,k_1,k_2,\cdots, k_m] ,\nonumber
\end{align}
with $|{\vec R}({\vec x})| \le C \alpha^{3+M}\| {\vec H}_0 \|_{W^{M+1,\infty}(B_\alpha)}$, $G( {\vec x} , {\vec z} ):=1/(4 \pi |{\vec x}- {\vec z}|)$. In the above, $J(p)$ and $K(m)$ are $p$-- and $m$--tuples of integers, respectively, with each element taking values $1,2,3$, and Einstein index summation is implied over $K(m+1)$ and $J(p+1)$.  Also
\begin{align}
 ( {\vec D}_{ x}^{2+m} G( {\vec x} , {\vec z} ))_{i, {K}(m+1)}  & =\left ( \prod_{\ell=1}^m  \partial_{x_{k_\ell}} \right ) (  \partial_{x_{k}} ( \partial_{x_{i} }( G( {\vec x} , {\vec z} )))) ,\nonumber \\
 ({\vec D}_z^p ({\vec H}_0  ( {\vec z} )))_{{J}(p+1)} & = \left (  \prod_{\ell=1}^p \partial_{z_{j_\ell}} \right )( {\vec H}_0({\vec z}) \cdot {\vec e}_j), \nonumber
\end{align}
and the coefficients of a rank $2+p+m$ generalised magnetic polarizability tensor (GMPT) are defined by
\begin{align*}
{\mathfrak M}_{{K}(m+1) {J}(p+1)}  :=& -{\mathfrak C}_{{K} (m+1) {J} (p+1) } + {\mathfrak N}_{ {K} (m+1) {J} (p+1)} , 
\end{align*}
where
\begin{subequations}, \label{eqn:checknten}
\begin{align}
{\mathfrak C}_{ {K} (m+1) {J} (p+1) } :=& - \frac{\im \nu \alpha^{3+m+p}(-1)^m}{2(m+1)! p! (p+2)} {\vec e}_k \cdot  \nonumber \\
&\int_B {\vec \xi} \times \left (  ( \Pi( {\vec \xi}))_{K(m)} ({\vec \theta}_{{J} (p+1) } + (\Pi({\vec \xi}))_{J(p)} {\vec e}_j \times {\vec \xi} )\right ) \dif {\vec \xi} ,  \\
{\mathfrak N}_{{K} (m+1) {J} (p+1) } : = & \left ( 1 - \mu_r^{-1} \right ) \frac{ \alpha^{3+m+p} (-1)^m}{p! m!} {\vec e}_k \cdot \nonumber \\
& \int_B (\Pi( {\vec \xi}))_{K(m)}  \left ( \frac{1}{p+2} \nabla_\xi \times {\vec \theta}_{{J}(p+1) } +(\Pi( {\vec \xi}))_{J(p)} {\vec e}_j  \right ) \dif {\vec \xi}. 
\end{align}
\end{subequations}
Furthermore, $ {\vec \theta}_{{J} (p+1)} $ satisfy the transmission problem
\begin{subequations} \label{eqn:trasprobten}
\begin{align}
\nabla_\xi \times \mu_r^{-1} \nabla_\xi \times {\vec \theta}_{ {J}(p+1) } - \im \nu {\vec \theta }_{ {J}(p+1) } & = \im \nu (\Pi({\vec \xi}))_{J(p)} {\vec e}_j \times {\vec \xi}  && \text{in $B$}, \\
\nabla_\xi \cdot {\vec \theta}_{ {J} (p+1) }  = 0 , \qquad \nabla_\xi \times  \nabla_\xi \times {\vec \theta}_{ {J}(p+1) }  & = {\vec 0} && \text{in $B^c:={\mathbb R}^3 \setminus \overline{B}$} , \\
[{\vec n} \times {\vec \theta}_{ {J} (p+1) }  ]_\Gamma &  = {\vec 0} && \text{on $\Gamma:=\partial B$},\\
  [{\vec n} \times \tilde{\mu}_r^{-1} \nabla_\xi \times {\vec \theta}_{{J} (p+1)}  ]_\Gamma & = -(p+2) [\tilde{\mu}_r^{-1 } ]_\Gamma ({\vec n} \times {\vec e}_j (\Pi ({\vec \xi}))_{J(p)} ) && \text{on $\Gamma$},\\
\int_\Gamma {\vec n} \cdot {\vec \theta}_{ {J} (p+1) } \dif {\vec \xi} & = 0 ,\\
{\vec \theta}_{ {J} (p+1) } & = O( | {\vec \xi} |^{-1}) && \text{as $|{\vec \xi} | \to \infty$}, 
\end{align}
\end{subequations}
$\displaystyle(\Pi({\vec \xi}))_{J(p)} := \prod_{\ell=1}^p  \xi_{j_\ell}= \xi_{j_1} \xi_{j_2} \cdots \xi_{j_p}$ and in the case $J(p)=\emptyset$ then $(\Pi({\vec \xi}))_{J(p)}=1$. Furthermore, $\tilde{\mu}({\vec \xi}):=\mu({\vec \xi}) /\mu_0$ so that $\tilde{\mu}_r=\mu_r$ for ${\vec \xi}\in B$ and $\tilde{\mu}_r=1$ otherwise, and $[\cdot ]_\Gamma= \cdot |_+ - \cdot |_-$ denotes the jump with $|_+$ denoting evaluation just outside of $\Gamma$ and $|_-$ just inside.
\end{theorem}

Note that, compared to~\cite{LedgerLionheart2018g}, we have chosen simplified the notation so that  
$\widecheck{\mathfrak C}$ is now written as ${\mathfrak C}$ and  $\widecheck{\widecheck{\mathfrak M}}$ as ${\mathfrak M}$.
\begin{remark}
In the case where $M=1$, (\ref{eqn:mainresult}) reduces to
\begin{align*}
({\vec H}_\alpha - {\vec H}_0) ({\vec x} )_i =&  ( {\vec D}_{ x}^{2} G( {\vec x} , {\vec z} ))_{ik} {\mathcal M}_{kj} (  {\vec H}_0  ( {\vec z} ))_{j} + ({\vec R}({\vec x}))_i, 
\end{align*}
where  $|{\vec R}({\vec x})| \le C \alpha^{4}\| {\vec H}_0 \|_{W^{2,\infty}(B_\alpha)}$ and ${\mathcal M}= {\mathcal M}_{kj} {\vec e}_k \otimes {\vec e}_j$ is the complex symmetric rank 2 MPT previously obtained in~\cite{LedgerLionheart2015} with alternative explicit expressions for ${\mathcal M}_{kj} $ derived in~\cite{LedgerLionheart2016} and~\cite{LedgerLionheart2019} agreeing with those of ${\mathfrak M}_{{K}(m+1) {J}(p+1)}$ in this case.
\end{remark}

\subsection{Multi-index form}

The asymptotic expansion presented in Theorem~\ref{thm:main} can  be alternatively obtained using a combination of tensor and multi-indices. This is achieved by using tensor indices, to reflect the vectorial nature of the problem, and multi-indices, to reflect summation over higher order derivatives of $ {\vec D}_{ x}^{2} G( {\vec x} , {\vec z} )$ and ${\vec H}_0  ( {\vec z})$. The alternative form is presented as the following result:

\begin{theorem}~\label{thm:mainmulti}
For any $M>0$, the magnetic field perturbation in the presence of a small conducting object $B_\alpha = \alpha B +{\vec z}$ for the eddy current model when $\nu$ and $\mu_r$ are order one and ${\vec x}$ is away from the location ${\vec z}$ of the inclusion is completely described by the asymptotic formula
\begin{align}
({\vec H}_\alpha - {\vec H}_0) ({\vec x} )_i =&  \sum_{\beta, |\beta|=0}^{M-1} \sum_{\delta,|\delta|=0}^{M-1-|\beta|} 
\partial_x^\beta (( {\vec D}_{ x}^{2} G( {\vec x} , {\vec z} ))_{ik})
 {\mathfrak M}_{kj}^{\beta \delta} \partial_z^\delta( ({\vec H}_0   ({\vec z} ))_{j}) + ({\vec R}({\vec x}))_i, \label{eqn:mainresultmult} 
\end{align}
with $|{\vec R}({\vec x})| \le C \alpha^{3+M}\| {\vec H}_0 \|_{W^{M+1,\infty}(B_\alpha)}$. In the above, $\beta=(\beta_1, \beta_2,\beta_3)$ and $\delta=(\delta_1,\delta_2,\delta_3)$ are multi-indices
and the coefficients of a rank $2+|\beta|+|\delta|$ generalised magnetic polarizability tensor (GMPT) are defined by
\begin{align*}
{\mathfrak M}_{kj}^{\beta\delta}  :=& -{\mathfrak C}_{kj}^{\beta\delta} + {\mathfrak N}_{kj}^{\beta\delta} , 
\end{align*}
where
\begin{align*}
{\mathfrak C}_{ kj}^{\beta\delta} :=& - \frac{\im \nu \alpha^{3+|\beta|+|\delta|}(-1)^{|\beta|}}{2\beta! \delta! (|\beta|+1) (|\delta|+2)} {\vec e}_k \cdot  
\int_B {\vec \xi} \times \left (  {\vec \xi}^\beta ({\vec \theta}_j^\delta + {\vec \xi}^\delta {\vec e}_j \times {\vec \xi} )\right ) \dif {\vec \xi} ,  \\
{\mathfrak N}_{kj }^{\beta\delta} : = & \left ( 1 -  \mu_r^{-1} \right ) \frac{ \alpha^{3+|\beta|+|\delta|} (-1)^{|\beta|}}{\beta! \delta!} {\vec e}_k \cdot  \int_B {\vec \xi}^\beta  \left ( \frac{1}{|\delta|+2} \nabla_\xi \times {\vec \theta}_{j }^\delta +{\vec \xi}^\delta {\vec e}_j  \right ) \dif {\vec \xi}. 
\end{align*}
Furthermore, $ {\vec \theta}_{j}^\delta $ satisfy the transmission problem
\begin{subequations} \label{eqn:transproblemmult}
\begin{align}
\nabla_\xi \times \mu_r^{-1} \nabla_\xi \times {\vec \theta}_{ j }^\delta - \im \nu  {\vec \theta }_{ j }^\delta & = \im \nu {\vec \xi}^\delta {\vec e}_j \times {\vec \xi}  && \text{in $B$}, \\
\nabla_\xi \cdot {\vec \theta}_{j }^\delta  = 0 , \qquad \nabla_\xi \times  \nabla_\xi \times {\vec \theta}_{ j }^\delta  & = {\vec 0} && \text{in $B^c$} , \\
[{\vec n} \times {\vec \theta}_{ j}^\delta  ]_\Gamma &  = {\vec 0} && \text{on $\Gamma$},\\
  [{\vec n} \times \tilde{\mu}_r^{-1} \nabla_\xi \times {\vec \theta}_{j } ^\delta ]_\Gamma & = -(|\delta|+2) [\tilde{\mu}_r^{-1 } ]_\Gamma ({\vec n} \times {\vec e}_j)  {\vec \xi}^\delta && \text{on $\Gamma$},\\
\int_\Gamma {\vec n} \cdot {\vec \theta}_{ j }^\delta \dif {\vec \xi} & = 0 ,\\
{\vec \theta}_{ j }^\delta & = O( | {\vec \xi} |^{-1}) && \text{as $|{\vec \xi} | \to \infty$}.
\end{align}
\end{subequations}
\end{theorem}

\begin{proof}
This result can be obtained by following similar steps to the proof of Theorem 4.1 in~\cite{LedgerLionheart2018g}, except instead of the form of the Taylor's series used in (23), (24) in~\cite{LedgerLionheart2018g}, the alternative forms
\begin{subequations} \label{eqn:ataylormulti}
\begin{align}
{\vec A}_0 (\alpha {\vec \xi} + {\vec z})  = &  \mu_0 \sum_{\beta, |\beta|=0}  \frac{\alpha^{1+|\beta|} } {\beta! ( |\beta|+2)} \partial_z^\beta (({\vec H}_0({\vec z}))_j) {\vec \xi}^\beta {\vec e}_j \times {\vec \xi} , \\
\nabla \times {\vec A}_0 (\alpha {\vec \xi} + {\vec z})  = & \mu_0 {\vec H}_0( \alpha {\vec \xi} + {\vec z})  = \mu_0  \sum_{\beta, |\beta|=0}  \frac{\alpha^{|\beta|} }{\beta! } \partial_z^\beta (({\vec H}_0({\vec z}))_j) {\vec \xi}^\beta {\vec e}_j ,
\end{align}
\end{subequations}
where $\beta=(\beta_1, \beta_2,\beta_3)$  are multi-indices.
Similarly, (47) and (48) in~\cite{LedgerLionheart2018g} are replaced by
\begin{subequations} \label{eqn:greentaylormulti}
\begin{align}
\nabla_x G({\vec x},{\vec y} ) = \sum_{\beta, |\beta|=0}^\infty \frac{(-1)^{|\beta|}}{\beta!} \partial_x^\beta(  \nabla_x (G({\vec x},{\vec z})
)) ( {\vec y}- {\vec z})^\beta ,  \\
{\vec D}_x^2  G({\vec x},{\vec y} ) = \sum_{\beta, |\beta|=0}^\infty \frac{(-1)^{|\beta|}}{\beta!} \partial_x^\beta(  {\vec D}_x^2  (G({\vec x},{\vec z})
)) ( {\vec y}- {\vec z})^\beta ,
\end{align}
\end{subequations}
then, by following the steps in~\cite{LedgerLionheart2018g}, we arrive at the alternative form of the asymptotic formula provided in  (43) as
\begin{align}
({\vec H}_\alpha - {\vec H}_0) ({\vec x} )=& - \im \nu \alpha^3   \sum_{\beta, |\beta|=0}^{M-1} \sum_{\delta,|\delta|=0}^{M-1-|\beta|}   \frac{(-1)^{|\beta|} \alpha^{|\beta| + |\delta| }}{\delta! \beta!  (|\beta | +1)(|\delta|+2) } \nonumber \\
&
  \int_B ( \partial_x^\beta ( {\vec D}_{ x}^{2} G( {\vec x} , {\vec z} )) {\vec \xi})  {\vec \xi}^\beta \times  ( {\vec \theta}_j^\delta + {\vec \xi}^\delta {\vec e}_j \times {\vec \xi} ) \dif {\vec \xi} \partial_z^\delta (({\vec H}_0({\vec z}))_j) \nonumber \\
& + \alpha^3 \left ( 1 - \mu_r^{-1} \right )  \sum_{\beta, |\beta|=0}^{M-1} \sum_{\delta,|\delta|=0}^{M-1-|\beta|}  \frac{(-1)^{|\beta|} \alpha^{|\beta| + |\delta| }}{\beta!\delta!} \partial_x^\beta (({\vec D}_x^2 G({\vec x},{\vec z}))_{ik}) {\vec e}_i \otimes {\vec e}_k \nonumber\\
& \int_B {\vec \xi}^\beta  \left ( \frac{1}{|\delta| +2 } \nabla \times {\vec \theta}_j^\delta + {\vec \xi}^\delta  {\vec e}_j 
 \right ) \dif {\vec \xi} \partial_z^\delta (({\vec H}_0({\vec z}))_j) + ({\vec R}({\vec x})) .  \label{eqn:multiexpand}
\end{align}
Then, by introducing,
\begin{align*}
({\vec H}_\alpha - {\vec H}_0) ({\vec x} )_i =  &  \sum_{\beta, |\beta|=0}^{M-1} \sum_{\delta,|\delta|=0}^{M-1-|\beta|}   \partial_x^\beta ( ({\vec D}_{ x}^{2} G( {\vec x} , {\vec z} ))_{\ell k} ) {\mathfrak A}_{i \ell k j}^{\beta \delta} \partial_z^\delta ({\vec H}_0({\vec z})_j) 
\nonumber \\
  & +\sum_{\beta, |\beta|=0}^{M-1} \sum_{\delta,|\delta|=0}^{M-1-|\beta|}   \partial_x^\beta (( {\vec D}_{ x}^{2} G( {\vec x} , {\vec z} ))_{i k}) {\mathfrak N}_{ k j}^{\beta \delta} \partial_z^\delta ({\vec H}_0({\vec z})_j) + ({\vec R}({\vec x}))_i,
\end{align*}
where
\begin{align}
 {\mathfrak A}_{i \ell k j}^{\beta \delta} : = & - \im \nu   \frac{(-1)^{|\beta|} \alpha^{3+ |\beta| + |\delta| }}{\delta! \beta ! (|\beta| +1) (|\delta|+2) } {\vec e}_i \cdot
 \int_B ( {\vec \xi})_\ell  {\vec \xi}^\beta {\vec e}_k \times  ( {\vec \theta}_j^\delta + {\vec \xi}^\delta {\vec e}_j \times {\vec \xi} ) \dif {\vec \xi} ,  \nonumber \\
{\mathfrak N}_{ k j}^{\beta \delta} : = &  \left ( 1 - \mu_r^{-1}  \right )    \frac{(-1)^{|\beta|} \alpha^{3+  |\beta| + |\delta| }}{\beta!\delta!} 
{\vec e}_k \cdot  \int_B {\vec \xi}^\beta  \left ( \frac{1}{|\delta| +2 } \nabla \times {\vec \theta}_j^\delta + {\vec \xi}^\delta  {\vec e}_j 
 \right ) \dif {\vec \xi} , \nonumber
\end{align}
and following similar steps to Lemma 6.3 and Lemma 6.4 in~\cite{LedgerLionheart2018g}  we obtain
\begin{align}
 {\mathfrak A}_{i \ell k j}^{\beta \delta}  = \epsilon_{ikr}  {\mathfrak C}_{r \ell j }^{\beta \delta},  \nonumber \\
  {\mathfrak C}_{r \ell j }^{\beta \delta}  = \epsilon_{\ell r k } \widecheck{\mathfrak C}_{ k j }^{\beta \delta} ,\nonumber 
 \end{align}
 where $\epsilon_{ijk}$ denotes the standard Levi-Cevita permutation symbol.
 Combining this with properties of ${\vec D}_x^2G({\vec x},{\vec z})$ leads to the final result.
\end{proof}

\subsubsection{Split field formulation}

In order to separately identify the contributions to the GMPTs associated with conducting and magnetic effects we derive the following.

\begin{theorem}~\label{thm:mainmultisplit} 
For any $M>0$, the magnetic field perturbation in the presence of a small conducting object $B_\alpha = \alpha B +{\vec z}$ for the eddy current model when $\nu$ and $\mu_r$ are order one and ${\vec x}$ is away from the location ${\vec z}$ of the inclusion is completely described by the asymptotic formula
\begin{align}
({\vec H}_\alpha - {\vec H}_0) ({\vec x} )_i =&  \sum_{\beta, |\beta|=0}^{M-1} \sum_{\delta,|\delta|=0}^{M-1-|\beta|} 
\partial_x^\beta (( {\vec D}_{ x}^{2} G( {\vec x} , {\vec z} ))_{ik}) {\mathfrak M}_{kj}^{\beta \delta} \partial_z^\delta(( {\vec H}_0  ( {\vec z} ))_{j}) + ({\vec R}({\vec x}))_i, \label{eqn:mainresultmultsplit} 
\end{align}
with $|{\vec R}({\vec x})| \le C \alpha^{3+M}\| {\vec H}_0 \|_{W^{M+1,\infty}(B_\alpha)}$. In the above, $\beta=(\beta_1, \beta_2,\beta_3)$ and $\delta=(\delta_1,\delta_2,\delta_3)$ are multi-indices
and the coefficients of a rank $2+|\beta|+|\delta|$ generalised magnetic polarizability tensor (GMPT) are defined by
\begin{align*}
{\mathfrak M}_{kj}^{\beta\delta}  :=& -({\mathfrak C}^{\sigma_*})_{kj}^{\beta\delta} +( {\mathfrak N}^{\sigma_*})_{kj}^{\beta\delta} + ( {\mathfrak N}^{0})_{kj}^{\beta\delta} , 
\end{align*}
where parenthesis have been used to make the presentation clearer and
\begin{subequations} \label{eqn:checkntensplit}
\begin{align}
({\mathfrak C}^{\sigma_*})_{ kj}^{\beta\delta} :=& - \frac{\im \nu \alpha^{3+|\beta|+|\delta|}(-1)^{|\beta|}}{2\beta ! \delta! (|\beta| +1) (|\delta|+2)} {\vec e}_k \cdot  
\int_B {\vec \xi} \times \left (  {\vec \xi}^\beta ({\vec \theta}_j^{(0),\delta} +{\vec \theta}_j^{(1),\delta} )
\right ) \dif {\vec \xi} ,  \\
({\mathfrak N}^{\sigma_*})_{kj }^{\beta\delta} : = & \left ( 1 - \mu_r^{-1}  \right ) \frac{ \alpha^{3+|\beta|+|\delta|} (-1)^{|\beta|}}{\beta! \delta!} {\vec e}_k \cdot  \int_B {\vec \xi}^\beta  \left ( \frac{1}{|\delta|+2} \nabla_\xi \times  {\vec \theta}_{j }^{(1),\delta}   \right )\dif {\vec \xi}, \\  
({\mathfrak N}^0)_{kj }^{\beta\delta} : = & \left ( 1 - \mu_r^{-1}   \right ) \frac{ \alpha^{3+|\beta|+|\delta|} (-1)^{|\beta|}}{\beta! \delta!} {\vec e}_k \cdot  \int_B {\vec \xi}^\beta  \left ( \frac{1}{|\delta|+2} \nabla_\xi \times  {\vec \theta}_{j }^{(0),\delta}  \right )\dif {\vec \xi}. 
\end{align}
\end{subequations}
In the above, $ {\vec \theta}_{j}^{(1),\delta} $ and $ {\vec \theta}_{j}^{(0),\delta} $ 
 satisfy the transmission problems
 \begin{subequations} \label{eqn:theta1tpmulti}
\begin{align}
\nabla_\xi \times \mu_*^{-1} \nabla_\xi \times {\vec \theta}_{ j }^{(1),\delta} - \im \nu  ( {\vec \theta }_{ j }^{(1),\delta} + {\vec \theta }_{ j }^{(0),\delta})  & =  {\vec 0}&& \text{in $B$}, \\
\nabla_\xi \cdot {\vec \theta}_{j }^{(1),\delta}  = 0 , \qquad \nabla_\xi \times   \nabla_\xi \times {\vec \theta}_{ j }^{(1),\delta}  & = {\vec 0} && \text{in $B^c$} , \\
[{\vec n} \times {\vec \theta}_{ j}^{(1), \delta}  ]_\Gamma &  = {\vec 0} && \text{on $\Gamma$},\\
  [{\vec n} \times \tilde{\mu}_r^{-1} \nabla_\xi \times {\vec \theta}_{j } ^{(1), \delta} ]_\Gamma & = {\vec 0} && \text{on $\Gamma$},\\
\int_\Gamma {\vec n} \cdot {\vec \theta}_{ j }^{(1),\delta} \dif {\vec \xi} & = 0 ,\\
{\vec \theta}_{ j }^{(1),\delta} & = O( | {\vec \xi} |^{-1}) && \text{as $|{\vec \xi} | \to \infty$}, 
\end{align}
\end{subequations}
and
\begin{subequations}
\begin{align}
\nabla_\xi \times \mu_r^{-1} \nabla_\xi \times {\vec \theta}_{ j }^{(0),\delta}   & =  {\vec 0}&& \text{in $B$}, \\
\nabla_\xi \cdot {\vec \theta}_{j }^{(0),\delta}  = 0 , \qquad \nabla_\xi \times  \nabla_\xi \times {\vec \theta}_{ j }^{(0),\delta}  & = {\vec 0} && \text{in $B^c$}, \\
[{\vec n} \times {\vec \theta}_{ j}^{(0), \delta}  ]_\Gamma &  = {\vec 0} && \text{on $\Gamma$},\\
  [{\vec n} \times \tilde{\mu}_r^{-1} \nabla_\xi \times {\vec \theta}_{j } ^{(0), \delta} ]_\Gamma & = {\vec 0} && \text{on $\Gamma$},\\
\int_\Gamma {\vec n} \cdot {\vec \theta}_{ j }^{(0),\delta} \dif {\vec \xi} & = 0 ,\\
{\vec \theta}_{ j }^{(0),\delta} - {\vec \xi}^\delta {\vec e}_j \times {\vec \xi}& = O( | {\vec \xi} |^{-1}) && \text{as $|{\vec \xi} | \to \infty$}, 
\end{align}
\end{subequations}
respectively.

\end{theorem}

\begin{proof}
We cannot set ${\vec \theta}_{j}^\delta = {\vec \theta}_j^{(1),\delta} + {\vec \theta}_j^{(0),\delta} - {\vec \xi}^\delta {\vec e}_j \times {\vec \xi}$ in Theorem~\ref{thm:mainmulti} since $\nabla_\xi  \cdot ( {\vec \xi}^\delta {\vec e}_j \times {\vec \xi}) \ne 0$ in general. Instead, we need to replace (39) in~\cite{LedgerLionheart2018g} with
\begin{align}
{\vec w}_0 ({\vec \xi}) = {\vec w}_0^{(1)} ({\vec \xi}) + {\vec w}_0^{(0)} ({\vec \xi}) - \sum_{\delta, |\delta |=0}^P \im \omega \mu_0 \frac{\alpha^{| \delta | } }{\delta!} \partial_z^\delta (({\vec H}_0({\vec z}) )_j) {\vec \xi}^\delta {\vec e}_j \times {\vec \xi} , \nonumber
\end{align}
where we can show that 
\begin{align*}
\nabla_\xi \cdot \left (  \sum_{\delta, |\delta |=0}^P \im \omega \mu_0 \frac{\alpha^{| \delta | } }{\delta!} \partial_z^\delta (({\vec H}_0({\vec z}))_j) {\vec \xi}^\delta {\vec e}_j \times {\vec \xi} \right )=0,
\end{align*}
which, instead of the transmission problem for ${\vec w}_0$, allows the introduction of the transmission problems for ${\vec w}_0^{(1)}$ and ${\vec w}_0^{(0)}$ as
\begin{align*}
\nabla_\xi \times \mu_r^{-1} \nabla_\xi \times ({\vec w}_0^{(0)} +{\vec w}_0^{(1)} ) - \im \nu({\vec w}_0^{(0)} +{\vec w}_0^{(1)} ) &= {\vec 0} && \text{in  } B,\\
\nabla_\xi \times  \nabla_\xi \times{\vec w}_0^{(1)}  &= {\vec 0} && \text{in  } B^c,\\
\nabla_\xi \cdot {\vec w}_0^{(1)} & = 0 && \text{in  } B^c, \\
 [ {\vec n } \times {\vec w}_0^{(1)} ]_\Gamma  = {\vec 0},  \qquad [ {\vec n } \times \tilde{\mu}_r^{-1} \nabla \times {\vec w}_0^{(1)}]_\Gamma &= {\vec 0} &&\text{on } \Gamma, \\
 {\vec w}_0^{(1)} & = O(|{\vec \xi} |^{-1} ) &&\text{as } |{\vec \xi}| \to \infty,
\end{align*}
and
\begin{align*}
\nabla_\xi \times \mu_r^{-1} \nabla_\xi \times {\vec w}_0^{(0)}  &= {\vec 0} && \text{in  } B,\\
\nabla_\xi \times  \nabla_\xi \times{\vec w}_0^{(0)}  &= {\vec 0} && \text{in  } B^c,\\
\nabla_\xi \cdot {\vec w}_0^{(0)} & = 0 && \text{in  } B\cup B^c, \\
 [ {\vec n } \times {\vec w}_0^{(0)} ]_\Gamma = {\vec 0}, \qquad [ {\vec n } \times \tilde{\mu}_r^{-1} \nabla \times {\vec w}_0^{(0)} ]_\Gamma &= {\vec 0} &&\text{on } \Gamma,\\
 {\vec w}_0^{(0)} -
 \left (  \sum_{\delta, |\delta |=0}^P \im \omega \mu_0 \frac{\alpha^{| \delta | } }{\delta!} \partial_z^\delta (({\vec H}_0({\vec z}))_j) {\vec \xi}^\delta {\vec e}_j \times {\vec \xi} \right )
  & = O(|{\vec \xi} |^{-1} ) &&\text{as } |{\vec \xi}| \to \infty,
\end{align*}
respectively. By introducing 
\begin{align}
{\vec w}_0^{(0)} = & \sum_{\delta, |\delta |=0}^P \im \omega \mu_0 \frac{\alpha^{| \delta | } }{\delta!} \partial_z^\delta (({\vec H}_0({\vec z}))_j) {\vec \theta}_j^{(0),\delta} , \nonumber\\
{\vec w}_0^{(1)} = & \sum_{\delta, |\delta |=0}^P \im \omega \mu_0 \frac{\alpha^{| \delta | } }{\delta!} \partial_z^\delta (({\vec H}_0({\vec z}))_j) {\vec \theta}_j^{(1),\delta}, \nonumber
\end{align}
and following similar steps to the proof of Theorem~\ref{thm:mainmulti} the result then follows.
\end{proof}

\begin{remark}
Theorem~\ref{thm:mainmultisplit} provides a natural extension of the splitting of an MPT, described in Lemma 1 of~\cite{LedgerLionheart2016}, to the case of GMPTs in terms of multi-indices.
\end{remark}

\section{GMPT properties} \label{sect:gmptprop}
In this section, we consider some properties of GMPTs including their symmetries, explicit formulae for their real and imaginary parts and also consideration of their spectral behaviour.

\subsection{GMPT symmetries}
Introducing 
\begin{align*}
{\mathfrak D}_{ {K} (m+1) {J} (p+1) } :=& (-1)^m {2(m+1)! p! (p+2)}
{\mathfrak C}_{ {K} (m+1) {J} (p+1) } \nonumber \\
 =  &- \im \nu \alpha^{3+m+p}{\vec e}_k \cdot
\int_B {\vec \xi} \times \left (  ( \Pi( {\vec \xi}))_{K(m)} ({\vec \theta}_{{J} (p+1) } + (\Pi({\vec \xi}))_{J(p)} {\vec e}_j \times {\vec \xi} )\right ) \dif {\vec \xi},
\end{align*}
then we have the following result on the symmetry of the tensor coefficients:
\begin{lemma} \label{lemma:symmetry}
For objects with $\mu_*=\mu_0$ the following symmetry holds
\begin{equation}
{\mathfrak D}_{ {K} (m+1) {J} (p+1) } = {\mathfrak D}_{ {J} (p+1) {K} (m+1) } .\label{eqn:symmetry}
\end{equation}
\end{lemma}
\begin{proof}
By using the transmission problem (\ref{eqn:trasprobten}) we get 
\begin{align*}
{\mathfrak D}_{ {K} (m+1) {J} (p+1) } =  & - \alpha^{3+m+p}
\int_B \im \nu  \left (   {\vec \theta}_{{J} (p+1) } + (\Pi({\vec \xi}))_{J(p)} {\vec e}_j \times {\vec \xi} \right )\cdot {\vec e}_k \times {\vec \xi}  ( \Pi( {\vec \xi}))_{K(m)} \dif {\vec \xi}\nonumber \\
=& - \alpha^{3+m+p} \int_B \nabla \times \mu_r^{-1} \nabla \times {\vec \theta}_{J(p+1)} \cdot 
\frac{1}{\im \nu} 
\left ( \nabla \times \mu_r^{-1} \nabla \times {\vec \theta}_{K(m+1) }- \im \nu  {\vec \theta}_{K(m+1) }
 \right ) \dif {\vec \xi}.
\end{align*}
Next, by applying integration by parts
\begin{align*}
\int_B \nabla \times \mu_r^{-1} \nabla \times {\vec \theta}_{J(p+1)} \cdot {\vec \theta}_{K(m+1) } \dif {\vec \xi} =&  \int_B  \mu_r^{-1}  \nabla \times {\vec \theta}_{J(p+1)} \cdot  \nabla \times {\vec \theta}_{K(m+1)} \dif {\vec \xi}  \nonumber  \\
& + \int_B  \nabla \cdot \left (   \mu_r^{-1} \nabla \times {\vec \theta}_{J(p+1)}  \times {\vec \theta}_{K(m+1)}   \right ) \dif {\vec \xi},
\end{align*}
and then using the transmission conditions in (\ref{eqn:trasprobten}) gives
\begin{align*}
\int_B & \nabla \cdot \left (   \mu_r^{-1} \nabla \times {\vec \theta}_{J(p+1)}  \times {\vec \theta}_{K(m+1)}   \right ) \dif {\vec \xi} =  \int_\Gamma {\vec \theta}_{K(m+1)} \cdot \left ( {\vec n}^-  \times \mu_r^{-1} \nabla \times {\vec \theta}_{J(p+1)}  \right ) |_- \dif {\vec \xi} \nonumber \\
 = & \int_\Gamma {\vec \theta}_{K(m+1)} \cdot \left ( {\vec n}^-  \times \nabla \times {\vec \theta}_{J(p+1)}  \right ) |_+ \dif {\vec \xi} + 
  \int_\Gamma (p+2) [\tilde{\mu}_r^{-1} ]_\Gamma {\vec \theta}_{K(m+1)} \cdot \left ( {\vec n}^- \times {\vec e}_j 
    \right )    (\Pi({\vec \xi}))_{J(p)} \dif {\vec \xi}
 \nonumber \\
=& - \int_{B^c} \nabla \cdot \left ( \nabla \times    {\vec \theta}_{J(p+1)} \times {\vec \theta}_{K(m+1)}  \right )  \dif {\vec \xi} + \int_B  (p+2) [\tilde{\mu}_r^{-1} ]_\Gamma  \nabla \cdot \left ( {\vec e}_j   (\Pi({\vec \xi}))_{J(p) }\times {\vec \theta}_{K(m+1)} \right) \dif {\vec \xi}\nonumber \\
= &\int_{B^c} \nabla \times  {\vec \theta}_{J(p+1)}  \cdot  \nabla \times  {\vec \theta}_{K(m+1)}  \dif {\vec \xi}    \nonumber \\
&+ \int_B  (p+2) [\tilde{\mu}_r^{-1} ]_\Gamma   \left ( {\vec \theta}_{K(m+1)}  \cdot  \nabla \times ({\vec e}_j (\Pi({\vec \xi}))_{J(p)} ) - \nabla \times  {\vec \theta}_{K(m+1)} \cdot {\vec e}_j (\Pi({\vec \xi}))_{J(p)}  \right ) \dif {\vec \xi} .
 \end{align*}
Considering the product $( {\vec D}_{ x}^{2+m} G( {\vec x} , {\vec z} ))_{i,{K}(m+1)} {\mathfrak D}_{ {K} (m+1) {J} (p+1) } $  and the above expression,  we have
 
 \begin{align*}
( {\vec D}_{ x}^{2+m} G( {\vec x} , {\vec z} ))_{i,{K}(m+1)}&  \int_B  \nabla \cdot \left (   \mu_r^{-1} \nabla \times {\vec \theta}_{J(p+1)}  \times {\vec \theta}_{K(m+1)}   \right ) \dif {\vec \xi} =\nonumber \\
& ( {\vec D}_{ x}^{2+m} G( {\vec x} , {\vec z} ))_{i,{K}(m+1)} \left ( \int_{B^c} \nabla \times  {\vec \theta}_{J(p+1)}  \cdot  \nabla \times  {\vec \theta}_{K(m+1)}  \dif {\vec \xi}  \right .\nonumber \\
& \left .-  \int_B  (p+2) [\tilde{\mu}_r^{-1} ]_\Gamma   \left ( \nabla \times  {\vec \theta}_{K(m+1)} \cdot {\vec e}_j (\Pi({\vec \xi}))_{J(p)}  \right ) \dif {\vec \xi} \right ).
 \end{align*}
 So that
 \begin{align*}
&{\mathfrak D}_{ {K} (m+1) {J} (p+1) } =   - \alpha^{3+m+p} \left ( \int_B \frac{1}{\im \nu} \nabla \times \mu_r^{-1} \nabla \times {\vec \theta}_{J(p+1)} \cdot \nabla \times \mu_r^{-1} \nabla \times {\vec \theta}_{K(m+1)}\dif {\vec \xi } \right . \nonumber\\
&\qquad  \left . - \int_{B\cup B^c} \tilde{\mu}_r^{-1} \nabla \times   {\vec \theta}_{J(p+1)}\cdot \nabla \times   {\vec \theta}_{K(m+1)}  \dif {\vec \xi } 
+  \int_B  (p+2) [\tilde{\mu}_r^{-1} ]_\Gamma  \left ( \nabla \times  {\vec \theta}_{K(m+1)} \cdot {\vec e}_j (\Pi({\vec \xi}))_{J(p)}  \right ) \dif {\vec \xi} \right ),
\end{align*}
with the required symmetry following for $\mu_*=\mu_0$.
 \end{proof}

\begin{corollary}
If using mutli-indices, we can introduce
\begin{align}
{\mathfrak D}_{ kj }^{\beta \delta} :=& (-1)^{|\beta|} {2\beta ! \delta! (|\beta| +1) (|\delta|+2)}
{\mathfrak C}_{ kj}^{\beta \delta}  \nonumber \\
 =  &- \im \nu \alpha^{3+|\beta|+|\delta| }{\vec e}_k \cdot
\int_B {\vec \xi} \times \left (   {\vec \xi}^\beta ({\vec \theta}_{j}^{\delta}  + {\vec \xi}^\delta {\vec e}_j \times {\vec \xi} )\right ) \dif {\vec \xi}, \label{eqn:dmult}
\end{align}
and following analogous steps to the proof of Lemma~\ref{lemma:symmetry} obtain
\begin{align*}
{\mathfrak D}_{ kj }^{\beta \delta} 
= & - \alpha^{3+m+p} \left ( \int_B \frac{1}{\im \nu} \nabla \times \mu_r^{-1} \nabla \times {\vec \theta}_{j}^\delta \cdot \nabla \times \mu_r^{-1} \nabla \times {\vec \theta}_{k}^{\beta }\dif {\vec \xi } \right . \nonumber\\
& \left . - \int_{B\cup B^c} \tilde{\mu}_r^{-1} \nabla \times   {\vec \theta}_{j}^\delta \cdot \nabla \times   {\vec \theta}_{k}^\beta  \dif {\vec \xi } 
+  \int_B  (|\delta|+2) [\tilde{\mu}_r^{-1} ]_\Gamma   \left ( \nabla \times  {\vec \theta}_{k}^\beta  \cdot {\vec e}_j ({\vec \xi})^\delta  \right ) \dif {\vec \xi} \right ),
\end{align*}
and, hence,
\begin{equation}
{\mathfrak D}_{kj}^{\beta \delta} = {\mathfrak D}_{jk }^{ \delta \beta} ,\label{eqn:symmetrymulit}
\end{equation}
for $\mu_* =\mu_0$.
\end{corollary}

\begin{corollary}
In the case of general $\mu_*$, 
\begin{align*}
{\mathfrak M}_{kj}^{\beta\delta}  :=& -{\mathfrak C}_{kj}^{\beta\delta} + {\mathfrak N}_{kj}^{\beta\delta} , \nonumber \\
= & \frac{ \alpha^{3+m+p}   (-1)^{|\beta|} }{ {\beta ! \delta! }}  \left ( \frac{1}{ 2 (|\beta| +1) (|\delta|+2)}
    \left ( \int_B \frac{1}{\im \nu} \nabla \times \mu_r^{-1} \nabla \times {\vec \theta}_{j}^\delta \cdot \nabla \times \mu_r^{-1} \nabla \times {\vec \theta}_{k}^\beta \dif {\vec \xi } \right . \right .  \nonumber\\
& \left . - \int_{B\cup B^c} \tilde{\mu}_r^{-1} \nabla \times   {\vec \theta}_{j}^\delta \cdot \nabla \times   {\vec \theta}_{k}^\beta   \dif {\vec \xi } 
+  \int_B  (|\delta|+2) [\tilde{\mu}_r^{-1} ]  \left ( \nabla \times  {\vec \theta}_{k}^\beta \cdot {\vec e}_j ({\vec \xi})^\delta  \right ) \dif {\vec \xi} \right )\nonumber\\
&+ [ \tilde{\mu}_r^{-1} ]  {\vec e}_k \cdot  \int_B ( {\vec \xi})^\beta    \left (  \frac{1}{|\delta | +2 } \nabla \times {\vec \theta}_j^\delta + {\vec \xi}^\delta  {\vec e}_j 
  \right ) \dif {\vec \xi}.
  \end{align*}
  Thus, we see that we have the symmetry
  \begin{align*}
  \frac{{\mathfrak M}_{kj}^{\beta\delta}}{(-1)^{|\beta|}} =   \frac{{\mathfrak M}_{jk}^{\delta\beta}}{(-1)^{|\delta|}} ={\mathfrak M}_{jk}^{\delta\beta},
  \end{align*}
  if $| \delta|= 2|\beta| $.
\end{corollary}

The analogous form of (\ref{eqn:dmult}) for the split fields is
\begin{align*}
{\mathfrak D}_{ kj }^{\beta \delta} :=& (-1)^{|\beta|} {2\beta ! \delta! (|\beta| +1) (|\delta|+2)}
{\mathfrak C}_{ kj}^{\beta \delta}  \nonumber \\
 =  &- \im \nu \alpha^{3+|\beta|+|\delta| }{\vec e}_k \cdot
\int_B {\vec \xi} \times \left (   {\vec \xi}^\beta ({\vec \theta}_{j}^{(0), \delta}  +{\vec \theta}_{j}^{(1), \delta}   )\right ) \dif {\vec \xi} ,
\end{align*}
and a related symmetry result for this case can be established also.

\begin{remark}
The symmetry properties  listed in this section extend the known complex symmetric property of rank 2 MPTs obtained in~\cite{LedgerLionheart2019,LedgerLionheart2015}.
\end{remark}

\subsection{Real and imaginary parts of GMPTs}

In this section, we establish explicit formulae for the real and imaginary parts of the coefficients of GMPTs through the following result:

\begin{lemma} \label{lemma:realimag}
For objects with $\mu_*=\mu_0$, the coefficients of ${\mathfrak D}_{ kj }^{\beta \delta}$ satisfy
\begin{align}
{\mathfrak D}_{ kj }^{\beta \delta} =& -  \alpha^{3+|\beta|+|\delta| } \left ( \int_B \frac{\im}{\nu} \nabla \times \mu_r^{-1} \nabla \times {\vec \theta}_{k}^{(1), \beta} \cdot \nabla \times \mu_r^{-1} \nabla \times \overline{{\vec \theta}_{j}^{(1),\delta}} \dif {\vec \xi } \right ) \nonumber \\
& \left . -
 \int_{B\cup B^c} \tilde{\mu}_r^{-1} \nabla \times   {\vec \theta}_{k}^{(1),\beta} \cdot \nabla \times  \overline{ {\vec \theta}_{j}^{(1),\delta} }  \dif  {\vec \xi }  \right ),
\label{eqn:complexsymmetry}
\end{align}
and can be written in the form ${\mathfrak D}_{ kj }^{\beta \delta} = {\mathfrak R}_{ kj }^{\beta \delta} + \im {\mathfrak I}_{ kj }^{\beta \delta}$ where
\begin{subequations}
\begin{align}
{\mathfrak R}_{ kj }^{\beta \delta}   = &   \text{Re} (   {\mathfrak D}_{ kj }^{\beta \delta}) =
   \alpha^{3+|\beta|+|\delta| }  \left (
\int_{B\cup B^c} \tilde{\mu}_r^{-1} \nabla \times {\vec \theta}_{k }^{ (1),\beta} \cdot   \nabla \times \overline{{\vec \theta}_{j }^{(1),\delta} } \dif {\vec \xi }  \right ),   \\
{\mathfrak I}_{ kj }^{\beta \delta}   = &   \text{Im} (   {\mathfrak D}_{ kj }^{\beta \delta}) = - \alpha^{3+|\beta|+|\delta| }  \left ( \int_B \frac{1 }{\nu} 
 \nabla \times \mu_r^{-1} \nabla \times {\vec \theta}_{k}^{(1),\beta} \cdot 
 \nabla \times \mu_r^{-1} \nabla \times \overline{{\vec \theta}_{j }^{(1),\delta} } \dif {\vec \xi }  \right ) ,
 \end{align}
 \end{subequations}
 and the overbar denotes the complex conjugate.
\end{lemma}
\begin{proof}
The first part of the proof follows similar steps to Lemma~\ref{lemma:symmetry}, but we note that since $\mu_*=\mu_0$ then ${\vec \xi}^\beta  {\vec e}_k \times {\vec \xi} = {\vec \theta}_k^{(0),\beta}$ and since ${\vec \theta}_k^{(0),\beta } \in {\mathbb R}^3$ we have
\begin{align}
{\vec \xi}^\beta  {\vec e}_k \times {\vec \xi} = & {\vec \theta}_k^{(0),\beta} =  \overline{{\vec \theta}_k^{(0),\beta }} =  \overline{ \frac{1}{\im\nu}} \left ( \nabla \times \mu_r^{-1} \nabla \times \overline{ {\vec \theta}_k^{(1),\beta}} -  \overline{ {\vec \theta}_k^{(1),\beta}}
  \right ). \nonumber
\end{align}
Thus, we obtain
\begin{align*}
{\mathfrak D}_{ kj }^{\beta \delta}  & =  -  \alpha^{3+|\beta|+|\delta| } \left ( \int_B \frac{\im }{\nu} 
 \nabla \times \mu_r^{-1} \nabla \times {\vec \theta}_{j}^{(1),\delta } \cdot 
 \nabla \times \mu_r^{-1} \nabla \times \overline{{\vec \theta}_{k }^{(1),\beta } } \dif {\vec \xi }  - 
\int_B \nabla \times \mu_r^{-1} \nabla \times {\vec \theta}_{j }^{ (1),\delta } \cdot   \overline{{\vec \theta}_{k }^{(1),\beta } } \dif {\vec \xi }   \right ) ,
\end{align*}
and a further application of integration parts gives (\ref{eqn:complexsymmetry}). Next, we proceed in an analogous way to the proof of Theorem 5.1 in~\cite{LedgerLionheart2019} and introduce the real and imaginary parts of ${\mathfrak D}_{ kj }^{\beta \delta} $  as
\begin{align*}
{\mathfrak R}_{ kj }^{\beta \delta}   = \text{Re} (   {\mathfrak D}_{ kj }^{\beta \delta}) = &    \alpha^{3+|\beta|+|\delta| } \text{Im} \left ( \int_B \frac{1 }{\nu} 
 \nabla \times \mu_r^{-1} \nabla \times {\vec \theta}_{j}^{(1),\delta } \cdot 
 \nabla \times \mu_r^{-1} \nabla \times \overline{{\vec \theta}_{k }^{(1),\beta } } \dif {\vec \xi }  \right ) \nonumber \\
 & +  \alpha^{3+|\beta|+|\delta| } \text{Re} \left (
\int_{B\cup B^c} \tilde{\mu}_r^{-1} \nabla \times {\vec \theta}_{j }^{ (1),\delta } \cdot   \nabla \times \overline{{\vec \theta}_{k }^{(1),\beta } } \dif {\vec \xi }   \right ) ,
\end{align*}
and
\begin{align*}
{\mathfrak I}_{ kj }^{\beta \delta}   =  \text{Im} (   {\mathfrak D}_{ kj }^{\beta \delta}) = &   - \alpha^{3+|\beta|+|\delta| } \text{Re} \left ( \int_B \frac{1 }{\nu} 
 \nabla \times \mu_r^{-1} \nabla \times {\vec \theta}_{j}^{(1),\delta } \cdot 
 \nabla \times \mu_r^{-1} \nabla \times \overline{{\vec \theta}_{k }^{(1),\beta } } \dif {\vec \xi }  \right ) \nonumber \\
 & +  \alpha^{3+|\beta|+|\delta| } \text{Im} \left (
\int_{B\cup B^c} \tilde{\mu}_r^{-1} \nabla \times {\vec \theta}_{j }^{ (1),\delta } \cdot   \nabla \times \overline{{\vec \theta}_{k }^{(1),\beta } } \dif {\vec \xi }   \right ) ,
\end{align*}
respectively.
Continuing to follow  the proof of Theorem 5.1 from~\cite{LedgerLionheart2019}, and by using properties of the complex conjugate and our earlier symmetry result (\ref{eqn:symmetrymulit}) for the tensor in multi-index form, we achieve the desired result.
\end{proof}

\begin{remark}
Lemma~\ref{lemma:realimag} shows that for $\mu_*=\mu_0$, explicit formulae for the real and imaginary parts of a GMPT can be obtained that are similar to those known for a rank 2 MPTs obtained in Theorem 5.1 of~\cite{LedgerLionheart2019}.
\end{remark}

\subsection{Spectral behaviour of GMPT coefficients}
The spectral behaviour of ${\vec \theta}_k^{(1), \beta}$ as a function of $\nu$ presented in the lemma below can be obtained in an analogous way to that of ${\vec \theta}_k^{(1)}$ derived in Lemma 8.2 of~\cite{LedgerLionheart2019}.
\begin{lemma}
The weak solution to (\ref{eqn:theta1tpmulti}) for $\nu \in [0,\infty)$ can be expressed as the convergent series
\begin{align*}
{\vec \theta}_k^{(1), \beta}  = - \sum_{n=1}^\infty \frac{\im \nu}{\im \nu - \lambda_n} P_n ( {\vec \theta}_k^{(0), \beta}) = \sum_{n=1}^\infty \beta_n P_n ( {\vec \theta}_k ^{(0), \beta}), \qquad \beta_n: = - \frac{\im \nu}{\im \nu - \lambda_n},
\end{align*}
where $P_n ( {\vec \theta}_k ^{(0), \beta})= {\vec \phi}_n \langle {\vec \theta}_k ^{(0),\beta}, {\vec \phi}_n \rangle_{L^2(B)}$, $\langle{\vec u},{\vec v}\rangle_{L^2(B)} : = \int_B {\vec u} \cdot \overline{\vec v} \dif {\vec \xi} $,  $(\lambda_n,{\vec \phi}_n)$ satisfy (39) in~\cite{LedgerLionheart2019} and 
\begin{align*}
\text{Re} (\beta_n) =  - \frac{ \nu^2}{\nu^2+\lambda_n^2}, \qquad \text{Im} (\beta_n) =   \frac{ \nu \lambda_n }{\nu^2+\lambda_n^2}.
\end{align*}
\end{lemma} 

Furthermore, by applying similar arguments to the proof of Lemma 8.5 in~\cite{LedgerLionheart2019}, we can also obtain the following result on the spectral behaviour of ${\mathfrak R}_{ kj }^{\beta \delta}$ and ${\mathfrak R}_{ kj }^{\beta \delta}$ with $\nu$:

\begin{lemma} \label{lemma:spectral}
The coefficients of ${\mathfrak R}_{ kj }^{\beta \delta}$ and ${\mathfrak I}_{ kj }^{\beta \delta}$  can be expressed as the convergent series
\begin{align}
{\mathfrak R}_{ kj }^{\beta \delta}   = &  -\frac{\alpha^{3+|\beta| + |\delta|}}{4} \sum_{n=1}^\infty \text{Re}(\beta_n) \lambda_n  \langle {\vec \phi}_n , {\vec \theta}_k^{(0),\beta } \rangle_{L^2(B)} \langle {\vec \phi}_n , {\vec \theta}_j^{(0), \delta} \rangle_{L^2(B)},  \nonumber \\
{\mathfrak I}_{ kj }^{\beta \delta}   = &  -\frac{\alpha^{3+|\beta| + |\delta|}}{4} \sum_{n=1}^\infty \text{Im}(\beta_n) \lambda_n \langle {\vec \phi}_n , {\vec \theta}_k ^{(0),\beta} \rangle_{L^2(B)} \langle {\vec \phi}_n , {\vec \theta}_j^{(0),\delta} \rangle_{L^2(B)} . \nonumber 
\end{align}
\end{lemma}
\begin{remark}
Lemma~\ref{lemma:spectral} shows that the spectral behaviour of the GMPT coefficients is very similar to that of the MPT coefficients previously obtained in Lemma 8.5 of~\cite{LedgerLionheart2019}. This has also been borne out in both the measurement and computation of GMPT coefficients that has been presented in~\cite{toykangmpt}.
\end{remark}

\section{Harmonic GMPTs} \label{sect:hgmpts}

For the purpose of this section, we assume that
\begin{itemize}
\item The object is located at the origin so that ${\vec z}={\vec 0}$.
\item The background ${\vec H}_0$ is generated by a small exciting coil at position ${\vec s}$ sufficiently far from the object so that it can be described as dipole source with moment ${\vec d}$ in the form
\begin{align}
({\vec H}_0 ({\vec 0}))_i = &({\vec D}_z^2 G({\vec z},{\vec s}))_{ij} ({\vec d})_j \nonumber \\
= &({\vec D}_x^2 G({\vec x},{\vec s}))_{ij} |_{{\vec x}={\vec 0}} ({\vec d})_j , \label{eqn:h0dipole}
\end{align}
at the position of the object with derivatives
\begin{align*}
\partial_z^\delta( ({\vec H}_0 ({\vec z}))_i)|_{{\vec z}={\vec 0}}  = & \partial_x^\delta (({\vec D}_x^2 G({\vec x},{\vec s}))_{ij}) |_{{\vec x}={\vec 0}}  ({\vec d})_j \nonumber \\
= & (-1)^{|\delta|}  \partial_s^\delta (({\vec D}_s^2 G({\vec s},{\vec 0}))_{ij} )  ({\vec d})_j .
\end{align*}

\end{itemize}
Application of these assumptions   to (\ref{eqn:mainresultmultsplit}) gives
\begin{align}
({\vec H}_\alpha - {\vec H}_0) ({\vec x} )_i =&  \sum_{\beta, |\beta|=0}^{M-1} \sum_{\delta,|\delta|=0}^{M-1-|\beta|} (-1)^{ |\delta|} 
\partial_x^\beta ( {\vec D}_{ x}^{2} G( {\vec x} , {\vec 0} ))_{ik}) {\mathfrak M}_{kj}^{\beta \delta}   \partial_s^\delta (({\vec D}_s^2 G({\vec s},{\vec 0}))_{j\ell } )  ({\vec d})_\ell    + ({\vec R}({\vec x}))_i. \label{eqn:assumpasyp}
\end{align}

\subsection{Green's function expansions}
Recall that  the Laplace free space Green's function $G({\vec x}, {\vec x}' )$,  where ${\vec x}$  and ${\vec x}'$  are the points with spherical coordinates $ (r,\theta,\psi)$ and $(r',\theta',\psi')$, respectively,  can also be expressed in terms of the (complex) spherical harmonics $Y_n^m(\theta,\psi) $ and $Y_n^m(\theta',\psi') $  of homogeneous degree $n$ and order $m$, with $-n \le m \le n$ and then in terms of the functions $K_n^m({\vec x}) = 1/r^{n+1} Y_n^m (\theta, \psi)$ and $H_n^m({\vec x}') = {r'}^n Y_n^m(\theta', \psi')$ as (e.g.~\cite{ledger2022})
 \begin{align}
 G({\vec x},{\vec x}')  = &  \sum_{n=0}^\infty  \frac{ |{\vec x}'|^n}{|{\vec x}|^{n+1}} \frac{1}{2n+1} \sum_{m=-n}^n Y_n^m(\theta,\phi) \overline{Y_n^m( \theta', \phi')}  \nonumber \\
 = &  \sum_{n=0}^\infty   \frac{1}{2n+1} \sum_{m=-n}^n K_n^m({\vec x}) \overline{H_n^m({\vec x}')} \label{eqn:gharmonic},
   \end{align}
provided that $|{\vec x }'|  < |{\vec x }| $. Noting that $H_n^m({\vec x}') $ are homogenous harmonic functions and that $K_n^m({\vec x})$ are also harmonic,  we observe that (\ref{eqn:gharmonic}) is harmonic with respect to ${\vec x}'$ and ${\vec x}$, respectively. Furthermore, $ G({\vec x},{\vec x}') $ can be expressed in terms of  real valued harmonic polynomials $I_n^{\ell} ({\vec x})$ and $I_n^{\ell'} ({\vec x}')$ as
\begin{align}
 G({\vec x},{\vec x}')  = & \sum_{n=0}^\infty   \frac{1}{2n+1} \frac{1}{{|\vec x }|^{2n+1}} \sum_{m=-n}^n  \sum_{\ell' = -n}^{n}  \sum_{\ell = -n}^{n}\overline{ a_{\ell ' m} ^{\IH}} I_{n }^{\ell'} ({\vec x}')  a_{\ell m} ^{\IH} I_{n }^{\ell} ({\vec x}) \nonumber \\
 = & \sum_{n=0}^\infty   \frac{1}{2n+1}  \frac{1}{{|\vec x }|^{2n+1}}    \sum_{\ell = -n}^{n}  I_{n }^{\ell} ({\vec x}') {I_{n}^{\ell} ({\vec x})} \nonumber ,
  \end{align}
since the coefficients $a_{\ell m} ^{\IH} $ satisfy $ \sum_{m=-n}^n \overline{a_{\ell ' m} ^{\IH} }  a_{\ell m} ^{\IH}  =\delta_{\ell' \ell}$, provided that $I_n^\ell({\vec x})$ are normalised appropriately~\cite{ledger2022}. Furthermore, from~\cite{ledger2022}, we have
\begin{align}
 G({\vec x},{\vec x}')  = \sum_{\beta, |\beta| =0}^ \infty  \frac{1}{2|\beta| +1} 
\sum_{m=-|\beta|}^{ | \beta|} K_{|\beta|}^m({\vec x})  \overline{a_{\beta m}^{\MH}} ({\vec x}')^\beta. \label{eqn:geenk}
\end{align}
For $| {\vec x}'|$ in a compact set  and as  $|{\vec x}| \to \infty$, a Taylor expansion~\cite{ammarikangbook}[pg. 77] gives
    \begin{equation}
 G({\vec x},{\vec x}')  = \sum_{\beta, |\beta| =0}^ \infty  \frac{(-1)^{|\beta|}}{\beta !} \partial_{x}^\beta G({\vec x},{\vec 0}) ( {\vec x}')^\beta \label{eqn:taylorg},
\end{equation}
so that 
\begin{align}
 \frac{1}{2|\beta| +1} 
\sum_{m=-|\beta|}^{ | \beta|} K_{|\beta|}^m({\vec x}) \overline{ a_{\beta m}^{\MH} }= & \frac{(-1)^{|\beta|}}{\beta !} \partial_{ x}^\beta G({\vec x},{\vec 0})  \label{eqn:taylortoh}.
\end{align}
Building on the above, we can also relate higher derivatives of ${\vec D}_x^2G({\vec x}, {\vec 0})$ to higher order derivatives of $K_{|\beta|}^m({\vec x})$  by differentiating (\ref{eqn:geenk}) term by term, since it is absolutely and uniformly convergent, giving
\begin{align}
{\vec D}_x^2  G({\vec x},{\vec x}')  = & \sum_{\beta, |\beta| =0}^ \infty  \frac{1}{2|\beta| +1} 
\sum_{m=-|\beta|}^{ | \beta|} {\vec D}_x ^2 (K_{|\beta|}^m({\vec x}))  \overline{a_{\beta m}^{\MH}} ({\vec x}')^\beta,  \label{eqn:deriv1}
\end{align}
and constructing a Taylor series expansion of ${\vec D}_x ^2 G({\vec x},{\vec x}')$  for $| {\vec x}'|$ in a compact set  as  $|{\vec x}| \to \infty$ in the form
\begin{align}
{\vec D}_x^2   G({\vec x},{\vec x}')  = &  \sum_{\beta, |\beta| =0}^ \infty  \frac{(-1)^{|\beta|}}{\beta !} \partial_{ x}^\beta ( {\vec D}_x^2 G({\vec x},{\vec 0})) ( {\vec x}')^\beta  \label{eqn:deriv2}.
\end{align} 
Thus,  by comparing (\ref{eqn:deriv1}) and (\ref{eqn:deriv2}),
\begin{align}
\frac{(-1)^{|\beta|}}{\beta !} \partial_{x}^\beta ( {\vec D}_x^2 G({\vec x},{\vec 0}) ) =   \frac{1}{2|\beta| +1} 
\sum_{m=-|\beta|}^{ | \beta|} \overline{ {\vec D}_x ^2 (  K_{|\beta|}^m({\vec x}) ) } {a_{\beta m}^{\MH}} = \frac{1}{2|\beta| +1} 
\sum_{m=-|\beta|}^{ | \beta|} { {\vec D}_x ^2 ( K_{|\beta|}^m({\vec x}) )}\overline  {a_{\beta m}^{\MH}}  ,\label{eqn:relate}
\end{align}
since ${\vec D}_x^2   G({\vec x},{\vec x}')$ is real.

\subsection{Harmonic GMPT expansion}

Using the alternative forms of Green function expansions allow us to introduce what we call  a harmonic GMPT (HGMPT) expansion for the assumptions listed in Section~\ref{sect:hgmpts}. The advantage of HGMPTs is that they require fewer coefficients than GMPTs to characterise an object for a given rank.

\begin{theorem} \label{thm:harmgmpt}
For any $M>0$, the magnetic field perturbation in the presence of a small conducting object $B_\alpha = \alpha B +{\vec z}$ for the eddy current model when $\nu $ and $\mu_r$ are order one and ${\vec x}$ is away from the location ${\vec z}$ of the inclusion under the assumptions in Section~\ref{sect:hgmpts} is completely described by the asymptotic formula
\begin{align}
({\vec H}_\alpha - {\vec H}_0) ({\vec x} )_i =& 
 \sum_{\ell =0}^{M-1} \sum_{t =0}^{M-1-\ell} 
   \sum_{p=-\ell}^\ell  \sum_{q=-t}^t   
  \left ({\vec D}_{ x}^{2}\left (  \frac{1}{|{\vec x}|^{2\ell+1}}
    I_{\ell}^p ({\vec x}) \right ) \right )_{ik}  
 { \mathfrak M}_{kj}^{{\rm H}, \ell p t q }   \left ( {\vec D}_s^2 \left ( \frac{1}{|{\vec s}|^{2t +1}}  I_{t}^q ({\vec s}) \right )\right )_{jo}
  d_o \nonumber \\
& + ({\vec R}({\vec x}))_i,  \label{eqn:hgmptas1}
\end{align}
with $|{\vec R}({\vec x})| \le C \alpha^{3+M}\| {\vec H}_0 \|_{W^{M+1,\infty}(B_\alpha)}$. In the above, $ {\mathfrak  M}_{kj}^{{\rm H} , \ell p t q }$ are the coefficients of  rank $ 2+\ell + t$ HGMPTs given by
\begin{align}
 {\mathfrak M}_{kj}^{{\rm H}, \ell p t q } =&
   \frac{\im \nu \alpha^{3+\ell +t}(-1)^{\ell }}{  2(\ell +1) (t+2) (2\ell+1)(2t+1)} {\vec e}_k \cdot  
\int_B {\vec \xi} \times \left (  I_\ell^p ({\vec \xi})  ({\vec \psi}_j^{(0),t,q} + {\vec \psi}_j^{(1),t,q } )\right ) \dif {\vec \xi}   \nonumber \\
 & +  \left ( 1 - \mu_r^{-1} \right )  \frac{ \alpha^{3+\ell +t} (-1)^{\ell }}{(2\ell+1)(2t+1)} {\vec  e}_k \cdot  \int_B I_\ell^p({\vec \xi})  \left ( \frac{1}{t+2} \nabla_\xi \times {\vec \psi}_{j }^{(1),t,q}   \right ) \dif {\vec \xi} \nonumber \\ 
 & + \left ( 1 - \mu_r^{-1}  \right ) \frac{ \alpha^{3+\ell +t} (-1)^{\ell }}{(2\ell+1)(2t+1)} {\vec  e}_k \cdot  \int_B I_\ell^p ({\vec \xi})   \left ( \frac{1}{t+2} \nabla_\xi \times {\vec \psi }_{j }^{(0), t,q}   \right ) \dif {\vec \xi},  
\end{align}
where ${\vec \psi}_j^{(0),t,q} $ and $ {\vec \psi}_j^{(1),t,q }$ satisfy
 the transmission problems 
\begin{subequations}  \label{eqn:harmonicpsi1}
\begin{align}
\nabla_\xi \times \mu_r^{-1} \nabla_\xi \times {\vec \psi}_j^{(0)  ,t,q } & = {\vec 0}  && \text{in $B$}, \\
\nabla_\xi \cdot {\vec \psi}_j^{ (0) ,  t,q  }  = 0 , \qquad \nabla_\xi \times  \nabla_\xi \times{\vec \psi}_j^{(0),  t,q }   & = {\vec 0} && \text{in $B^c$}, \\
[{\vec n} \times{\vec \psi}_j^{ (0) ,t,q }  ]_\Gamma &  = {\vec 0} && \text{on $\Gamma$},\\
  [{\vec n} \times \tilde{\mu}_r ^{-1} \nabla_\xi \times {\vec \psi}_j^{ (0), t,q  }  ]_\Gamma & ={\vec 0}  && \text{on $\Gamma$},\\
{\vec \psi}_j^{(0),  t,q } -  I_t^q( {\vec \xi}) {\vec e}_j \times {\vec \xi}
 & = O( | {\vec \xi} |^{-1}) && \text{as $|{\vec \xi} | \to \infty$},
\end{align}
\end{subequations}
and
\begin{subequations} \label{eqn:harmonicpsi0}
\begin{align}
\nabla_\xi \times \mu_r^{-1} \nabla_\xi \times {\vec \psi}_j^{(1)  ,t,q  }- \im \nu ( {\vec \psi}_j^{(1) ,t ,s  } + {\vec \psi}_j^{(0) ,t,q  } ) & = {\vec 0}  && \text{in $B$}, \\
\nabla_\xi \cdot {\vec \psi}_j^{ (1)  ,t,q }  = 0 , \qquad \nabla_\xi \times  \nabla_\xi \times{\vec \psi}_j^{(1),t,q }   & = {\vec 0} && \text{in $B^c$}, \\
[{\vec n} \times{\vec \psi}_j^{ (1), t,q }  ]_\Gamma &  = {\vec 0} && \text{on $\Gamma$},\\
  [{\vec n} \times \tilde{\mu}_r^{-1} \nabla_\xi \times {\vec \psi}_j^{ (1) ,t ,q  }  ]_\Gamma & ={\vec 0}  && \text{on $\Gamma$},\\
\int_\Gamma {\vec n} \cdot{\vec \psi}_j^{(1)  ,t ,q  }  \dif {\vec \xi} & = 0,\\
{\vec \psi}_j^{(1)  ,t,q }  & = O( | {\vec \xi} |^{-1}) && \text{as $|{\vec \xi} | \to \infty$}.
\end{align}
\end{subequations}
\end{theorem}
\begin{proof}
Starting from (\ref{eqn:assumpasyp}) and using (\ref{eqn:relate}) we get
\begin{align}
({\vec H}_\alpha - {\vec H}_0) ({\vec x} )_i =&  \sum_{\beta, |\beta|=0}^{M-1} \sum_{\delta,|\delta|=0}^{M-1-|\beta|} \frac{\beta ! \delta !}{(2|\beta|+1)(2|\delta |+1)} 
\sum_{m=-|\beta|}^{|\beta|}\sum_{n=-|\delta|}^{|\delta|} 
\nonumber \\
&\qquad (  {\vec D}_{ x}^{2}( {K_{|\beta|}^m (  {\vec x} ) }))_{ik}  \overline{a_{\beta m}^{\MH}}  {\mathfrak M}_{kj}^{\beta \delta}(  {\vec D}_s^2 ( {K_{|\delta|}^n ( {\vec s} ) })) _{jp} d_p  \overline{a_{\delta n}^{\MH} }+ ({\vec R}({\vec x}))_i, \nonumber \\
=&  \sum_{\ell =0}^{M-1} \sum_{t =0}^{M-1-m}
\sum_{m=-\ell} ^{\ell }\sum_{n=-t }^{t } 
 ({\vec D}_{ x}^{2}( {K_{\ell }^m (  {\vec x}  )}))_{ik}   {\mathfrak M}_{kj}^{{\rm C}, \ell m t n } ( {\vec D}_s^2 ( {K_{ t}^n ( {\vec s} )}))_{jo}  d_o  + ({\vec R}({\vec x}))_i, \label{eqn:preharmonicsymp1}
\end{align}
 where, unlike in~\cite{ledger2022}, we do not choose to take the complex conjugate of $ {\vec D}_x^2 (  K_{ t}^n ( {\vec s} ))\overline{a_{\delta n}^{\MH} }$ since the {contracted} type GMPTs $ {\mathfrak M}^{\rm C
 }$ are themselves complex, and have coefficients
\begin{align*}
  {\mathfrak M}_{kj}^{{\rm C},\ell m t n }  = & \frac{1}{(2\ell +1)(2t +1)}   \sum_{\beta, |\beta| = \ell} \sum_{\delta, |\delta | =t } \beta ! \delta ! \overline{a_{\beta m}^{\MH}} {\mathfrak M}_{kj}^{\beta \delta} \overline{a_{\delta n}^{\MH}} \nonumber \\
 =& \frac{1}{(2\ell +1)(2t +1)}   \sum_{\beta, |\beta| = \ell} \sum_{\delta, |\delta | =t } \left ( \overline{a_{\beta m}^{\MH} }  \frac{\im \nu \alpha^{3+|\beta|+|\delta|}(-1)^{|\beta|}}{ 2 (|\beta|+1) (|\delta|+2)} {\vec e}_k \cdot  
\int_B {\vec \xi} \times \left (  {\vec \xi}^\beta ({\vec \theta}_j^{(0),\delta} + {\vec \theta}_j^{(1),\delta} )\right ) \dif {\vec \xi} \overline{a_{\delta n}^{\MH}}  \right . \nonumber \\
 & +\overline{a_{\beta m}^{\MH} }  \left ( 1 - \mu_r^{-1} \right )  \alpha^{3+|\beta|+|\delta|} (-1)^{|\beta|} {\vec  e}_k \cdot  \int_B {\vec \xi}^\beta  \left ( \frac{1}{|\delta|+2} \nabla_\xi \times {\vec \theta}_{j }^{(1),\delta}   \right ) \dif {\vec \xi}\overline{ a_{\delta n}^{\MH}} \nonumber \\ 
 & \left .  +\overline{a_{\beta m}^{\MH}}  \left ( 1 - \mu_r^{-1}  \right )  \alpha^{3+|\beta|+|\delta|} (-1)^{|\beta|} {\vec  e}_k \cdot  \int_B {\vec \xi}^\beta  \left ( \frac{1}{|\delta|+2} \nabla_\xi \times {\vec \theta}_{j }^{(0),\delta}   \right ) \dif {\vec \xi} \overline{ a_{\delta n}^{\MH} } \right )  \nonumber \\  
 =& \frac{1}{(2\ell +1)(2t +1)}    \left (   \frac{\im \nu \alpha^{3+\ell +t}(-1)^{\ell }}{  2(\ell +1) (t +2)} {\vec e}_k \cdot  
\int_B {\vec \xi} \times \left (  \overline{H_\ell^m({\vec \xi}) } ({\vec \phi}_j^{(0),t,n} + {\vec \phi}_j^{(1),t,n} )\right ) \dif {\vec \xi}    \right . \nonumber \\
 & + \left ( 1 - \mu_r^{-1}  \right )  \alpha^{3+\ell+t} (-1)^{\ell } {\vec  e}_k \cdot  \int_B \overline{H_\ell^m {\vec \xi}  }\left ( \frac{1}{t+2} \nabla_\xi \times {\vec \phi}_{j }^{(1),t,n}   \right ) \dif {\vec \xi}  \nonumber \\ 
 & \left . +  \left ( 1 - \mu_r^{-1} \right )  \alpha^{3+\ell+t } (-1)^{\ell } {\vec  e}_k \cdot  \int_B \overline{H_\ell^m (  {\vec \xi})  }  \left ( \frac{1}{t+2} \nabla_\xi \times {\vec \phi}_{j }^{(0),t,n}   \right ) \dif {\vec \xi} \right )   .
\end{align*}
In the above, we have used $\sum_{\beta, |\beta| = \ell } a _{\beta m}^{\MH} {\vec \xi}^ \beta = H_\ell^m({\vec \xi})$. The vector fields ${\vec \phi}_j^{ (0),t , n  }({\vec \xi})$ and ${\vec \phi}_j^{ (1),t  , n  }({\vec \xi})$ satisfy the transmission problems 
\begin{subequations}  \label{eqn:harmonicphi0}
\begin{align}
\nabla_\xi \times \mu_r^{-1} \nabla_\xi \times {\vec \phi}_j^{ (0),t , n  } & = {\vec 0}  && \text{in $B$}, \\
\nabla_\xi \cdot {\vec \phi}_j^{ (0),t , n  }  = 0 , \qquad \nabla_\xi \times  \nabla_\xi \times{\vec \phi}_j^{ (0),t , n  }   & = {\vec 0} && \text{in $B^c$} , \\
[{\vec n} \times{\vec \phi}_j^{ (0),t , n  } ]_\Gamma &  = {\vec 0} && \text{on $\Gamma$},\\
  [{\vec n} \times \tilde{\mu}_r^{-1} \nabla_\xi \times {\vec \phi}_j^{ (0),t , n  } ]_\Gamma & ={\vec 0}  && \text{on $\Gamma$},\\
{\vec \phi}_j^{ (0),t , n  }-  \overline{H_t^n( {\vec \xi}) } {\vec e}_j \times {\vec \xi}
 & = O( | {\vec \xi} |^{-1}) && \text{as $|{\vec \xi} | \to \infty$},
\end{align}
\end{subequations}
and
\begin{subequations} \label{eqn:harmonicphi1}
\begin{align}
\nabla_\xi \times \mu_r^{-1} \nabla_\xi \times {\vec \phi}_j^{ (1),t , n  }- \im \nu ( {\vec \phi}_j^{ (1),t , n  } + {\vec \phi}_j^{ (0),t , n  } ) & = {\vec 0}  && \text{in $B$}, \\
\nabla_\xi \cdot{\vec \phi}_j^{ (1),t , n  }  = 0 , \qquad \nabla_\xi \times  \nabla_\xi \times {\vec \phi}_j^{ (1),t , n  }  & = {\vec 0} && \text{in $B^c$} , \\
[{\vec n} \times{\vec \phi}_j^{ (1),t , n  } ]_\Gamma &  = {\vec 0} && \text{on $\Gamma$},\\
  [{\vec n} \times \tilde{\mu}_r^{-1} \nabla_\xi \times{\vec \phi}_j^{ (1),t , n  }  ]_\Gamma & ={\vec 0}  && \text{on $\Gamma$},\\
\int_\Gamma {\vec n} \cdot{\vec \phi}_j^{ (1),t , n  } \dif {\vec \xi} & = 0 ,\\
{\vec \phi}_j^{ (1),t , n  } & = O( | {\vec \xi} |^{-1}) && \text{as $|{\vec \xi} | \to \infty$},
\end{align}
\end{subequations}
respectively. Recall from~\cite{ledger2022} that
\begin{align}
\sum_{\beta, |\beta| = \ell } a _{\beta m}^{\MH} {\vec \xi}^ \beta = H_\ell^m({\vec \xi}) = \sum_{u=-\ell}^\ell  a_{u m}^{\IH} I_\ell^u( {\vec \xi}) ,\label{eqn:IHrelation}
\end{align}
are harmonic functions, and  that $ I_\ell^u( {\vec \xi})$ are real valued, then we can also write
\begin{align}
\sum_{\beta, |\beta | = \ell }  \overline{ a_{\beta m}^{\MH} } {\vec \theta}_k^{(1), \beta } ({\vec \xi}) = &  {\vec \phi}_k^{(1), \ell,m } ({\vec \xi}) 
= \sum_{u=-\ell}^\ell  \overline{ a_{um}^{\IH} } {\vec \psi}_k^{(1), \ell , u  }({\vec \xi}), \nonumber \\
\sum_{\beta, |\beta | = \ell }   \overline{ a_{\beta m}^{\MH} } {\vec \theta}_k^{(0), \beta } ({\vec \xi}) = & {\vec \phi}_k^{(0), \ell,m } ({\vec \xi}) =
 \sum_{u=-\ell}^\ell  \overline{a_{um}^{\IH} } {\vec \psi}_k^{(0), \ell , u  }({\vec \xi}), \nonumber 
\end{align}
where, after an appropriate replacement of indices,  ${\vec \psi}_k^{ (0),\ell , u  }({\vec \xi})$ and ${\vec \psi}_k^{ (1),\ell , u  }({\vec \xi})$ are solutions to the transmission problems (\ref{eqn:harmonicpsi1}) and (\ref{eqn:harmonicpsi0}), respectively.
This means  that
\begin{align}
  {\mathfrak M}_{kj}^{{\rm C}, \ell m t n }  = &   \sum_{u=-\ell}^\ell \sum_{v=-t }^t  \overline{a_{u m}^{\IH} }{\mathfrak  M}_{kj}^{{\rm H}, \ell u t v } \overline{ a_{v n }^{\IH} } \label{eqn:prehgmpt} \\
  = & \sum_{u=-\ell}^\ell \sum_{v=-t }^t  \overline{a_{u m}^{\IH} }
   \frac{\im \nu \alpha^{3+\ell +t}(-1)^{\ell }}{2  (\ell +1) (t+2) (2\ell +1)(2t+1)} {\vec e}_k \cdot  
\int_B {\vec \xi} \times \left (  I_\ell^u ({\vec \xi})  ({\vec \psi}_j^{(0),t,v } + {\vec \psi}_j^{(1),t,v  } )\right ) \dif {\vec \xi}  \overline{ a_{v n}^{\IH} } \nonumber \\
 & +\overline{a_{u m}^{\IH}  } \left ( 1 - \mu_r^{-1} \right ) \frac{ \alpha^{3+\ell +t} (-1)^{\ell }}{(2\ell+1)(2t+1)} {\vec  e}_k \cdot  \int_B I_\ell^u ({\vec \xi})  \left ( \frac{1}{t+2} \nabla_\xi \times {\vec \psi}_{j }^{(1),t,v}   \right ) \dif {\vec \xi} \overline{ a_{v n}^{\IH} } \nonumber \\ 
 & +\overline{ a_{u m}^{\IH} } \left ( 1 - \mu_r^{-1}  \right )  \frac{\alpha^{3+\ell +t} (-1)^{\ell }}{(2\ell+1)(2t+1)} {\vec  e}_k \cdot  \int_B I_\ell^u ({\vec \xi})   \left ( \frac{1}{t+2} \nabla_\xi \times {\vec \psi }_{j }^{(0), t,v}   \right ) \dif {\vec \xi} \overline{ a_{v n}^{\IH}} ,  
\end{align}
and $ {\mathfrak M}_{kj}^{{\rm H}, \ell u t v  }$ are coefficients of what we call harmonic GMPTs (HGMPTs). We then introduce  (\ref{eqn:prehgmpt})  into 
 (\ref{eqn:preharmonicsymp1})  leading to
\begin{align}
({\vec H}_\alpha - {\vec H}_0) ({\vec x} )_i =& 
 \sum_{\ell =0}^{M-1} \sum_{t =0}^{M-1-\ell} 
\sum_{m=-\ell} ^{\ell }\sum_{n=-t }^{t } 
 \sum_{u=-\ell}^\ell \sum_{v=-t }^t  
 ({\vec D}_{ x}^{2}(  {K_{\ell }^m (  {\vec x}  )}))_{ik} \overline{a_{u m}^{\IH}  } {\mathfrak M}_{kj}^{{\rm H},\ell u t v } \overline{ a_{vn }^{\IH}} \nonumber \\
&  ( {\vec D}_s^2 (  {K_{ t}^n ( {\vec s} )}))_{jo}  d_o  \label{eqn:hgmptas2}
 + ({\vec R}({\vec x}))_i. 
\end{align}
Now, by using (\ref{eqn:IHrelation}), we can write 
\begin{align}
K_\ell^m({\vec x}) = \frac{1}{|{\vec x}|^{2\ell+1}} H_{\ell}^m ({\vec x}) = \frac{1}{|{\vec x}|^{2\ell +1}} \sum_{p=-\ell}^\ell  a_{pm}^{\IH} I_{\ell}^p ({\vec x}),
\nonumber
\end{align}
and substituting into (\ref{eqn:hgmptas2}) we get
\begin{align*}
({\vec H}_\alpha - {\vec H}_0) ({\vec x} )_i =& 
 \sum_{\ell =0}^{M-1} \sum_{t =0}^{M-1-\ell} 
\sum_{m=-\ell} ^{\ell }\sum_{n=-t }^{t } 
 \sum_{u=-\ell}^\ell \sum_{v=-t }^t  \sum_{p=-\ell}^\ell  \sum_{q=-t}^t   \nonumber \\
 & \left ({\vec D}_{ x}^{2}\left (  \frac{1}{|{\vec x}|^{2\ell +1}}  I_{\ell}^p ({\vec x}) \right ) \right )_{ik}  {a_{pm}^{\IH} }
 \overline{a_{u m}^{\IH}} { \mathfrak M}_{kj}^{{\rm H}, \ell u t v } \nonumber \\
&  \left ( {\vec D}_s^2 \left ( \frac{1}{|{\vec s}|^{2t +1}}  I_{t}^q ({\vec s}) \right )\right )_{jo }
\overline{ a_{vn }^{\IH} } {a_{qn}^{\IH}}
  d_o  
 + ({\vec R}({\vec x}))_i.
\end{align*}
Finally, using
\begin{align*}
\sum_{m=-\ell}^{\ell} {a_{pm}^{\IH}} \overline{a_{um}^{\IH} } = \delta_{pu}, 
\end{align*}
completes the proof.
\end{proof}

\begin{corollary} \label{coll:vsr}
It immediately follows from Theorem~\ref{thm:harmgmpt} that the voltage induced in a small  receiving coil at position ${\vec x}^{\rm r}$ with dipole moment ${\vec f}$ due to a small source coil at position ${\vec x}^{\rm s}$ with dipole moment ${\vec d}$, after truncation, is
\begin{align*}
V_{{\rm{s} \rm{r}}}=({\vec H}_\alpha - {\vec H}_0) ({\vec x}^{\rm{r}} )_i ({\vec f})_i =& 
 \sum_{\ell =0}^{M-1} \sum_{t =0}^{M-1-\ell}  
   \sum_{p=-\ell}^\ell  \sum_{q=-t}^t   
 {f}_i   \left (\left . {\vec D}_{ x}^{2}\left (  \frac{1}{|{\vec x } |^{2\ell +1}}  I_{\ell}^p ({\vec x}) \right ) \right |_{{\vec x}= {\vec x}^{\rm{r}}} \right )_{ik}  
 {\mathfrak  M}_{kj}^{{\rm H}, \ell p t q } \nonumber \\
&  \left ( \left . {\vec D}_x^2 \left ( \frac{1}{|{\vec x} |^{2t +1}}  I_{t}^q ({\vec x} ) \right ) \right |_{{\vec x}={\vec x}^{\rm{s}}}  \right )_{jo}
  d_o ,
\end{align*}
where we note the use of Roman ${\rm{r}}$ and ${\rm{s}}$ to denote the receive and source, respectively.
\end{corollary}

For what follows, we define the HGMPT matrices ${\mathbf C}_{kj}^{{\rm H}, \ell t}$  and ${\mathbf N}_{kj}^{{\rm H}, \ell t}$
 with coefficients
\begin{align*}
( {\mathbf C}_{kj}^{{\rm H},\ell t})_{pq} : =&-
\im \nu    {\vec e}_k \cdot  
\int_B {\vec \xi} \times \left (  I_\ell^p ({\vec \xi})  ({\vec \psi}_j^{(0),t,q} + {\vec \psi}_j^{(1),t,q } )\right ) \dif {\vec \xi},  \\
( {\mathbf N}_{kj}^{{\rm H},\ell t})_{pq} := &  \left ( 1 - \mu_r^{-1} \right )   {\vec  e}_k \cdot  \int_B I_\ell^p({\vec \xi})  \left (  \nabla_\xi \times {\vec \psi}_{j }^{(1),t,q}   \right ) \dif {\vec \xi} \nonumber \\
 & + \left ( 1 - \mu_r^{-1}  \right )  {\vec  e}_k \cdot  \int_B I_\ell^p ({\vec \xi})   \left (  \nabla_\xi \times {\vec \psi }_{j }^{(0), t,q}   \right ) \dif {\vec \xi},  
\end{align*}
which are of dimension $(2\ell+1)\times(2t+1)$,
so that
\begin{align*}
{ \mathfrak M}_{kj}^{{\rm H}, \ell p t q } = -\frac{ \alpha^{3+\ell +r}(-1)^{\ell }}{  2(\ell +1) (t+2) (2\ell+1)(2t+1)} ( {\mathbf C}_{kj}^{{\rm H},\ell t})_{pq} + \frac{ \alpha^{3+\ell +t} (-1)^{\ell }}{(2\ell+1)(2t+1)} \frac{1}{t+2} ( {\mathbf N}_{kj}^{{\rm H},\ell t})_{pq} ,
\end{align*} 
for $ -\ell \le p \le \ell , -t \le q \le t $.
Additionally, we define the  $1\times (2t+1)$ and $1 \times (2\ell+1)$ matrices ${\mathbf D}{\mathbf I}_{{\mathrm{s}j}}^{t}$ and ${\mathbf D}{\mathbf I}_{{\mathrm{r}k}}^{\ell}$  with coefficients
\begin{align}
({\mathbf D}{\mathbf I}_{\mathrm{s} j}^{t})_q  &: = \sum_{o=1}^3 \left ( \left . {\vec D}_x^2 \left ( \frac{1}{|{\vec x} |^{2t +1}}  I_{t}^q ({\vec x}) \right )\right |_{{\vec x}=  {\vec x}^{\rm s}}  \right )_{jo} \nonumber
  d_o,  \qquad -t \le q \le t , \\
( {\mathbf D}{\mathbf I}_{{\mathrm{r}k}}^{\ell})_p  & := \sum_{i=1}^3   \left (\left . {\vec D}_{ x}^{2}\left (  \frac{1}{|{\vec x} |^{2\ell +1}}  I_{\ell}^p ({\vec x}) \right ) \right |_{{\vec x}=  {\vec x} ^{\rm r}} \right )_{ik}  {f}_i ,
\qquad  -\ell \le p \le \ell ,
\nonumber
\end{align}
where again the Roman ${\rm r}$ and ${\rm s}$ denote receiver and source, respectively.
It then follows from Corollary~\ref{coll:vsr} that, after truncation,
\begin{align}
V_{{\rm{s} \rm{r}}}  = & - \sum_{\ell =0}^{M-1} \sum_{t =0}^{M-1-\ell} \frac{  \alpha^{3+\ell +t}(-1)^{\ell }}{  2(\ell +1) (t+2) (2\ell+1)(2t+1)}    \sum_{k,j=1}^3  {\mathbf D}{\mathbf I}_{{\mathrm{r}k}}^{\ell}  {\mathbf C}_{kj}^{{\rm H}, \ell t} ( {\mathbf D}{\mathbf I}_{\mathrm{s} j}^{t} )^T \nonumber \\
&+\sum_{\ell =0}^{M-1} \sum_{t =0}^{M-1-\ell}  \frac{ \alpha^{3+\ell +t} (-1)^{\ell }}{(2\ell+1)(2t+1)}  \frac{1}{t+2}   \sum_{k,j=1}^3 {\mathbf D}{\mathbf I}_{{\mathrm{r}k}}^{\ell}  {\mathbf N }_{kj}^{{\rm H}, \ell t} ( {\mathbf D}{\mathbf I}_{\mathrm{s} j}^{t} )^T ,
 \label{eqn:indvoltharm}
\end{align}
where $T$ denotes the transpose. 
An alternative description of $V_{{\rm{s} \rm{r}}}$ follows from (\ref{eqn:preharmonicsymp1}) by introducing the matrices
${\mathbf C}_{kj}^{{\rm C}, \ell r}$  and ${\mathbf N}_{kj}^{{\rm C}, \ell r}$
 with coefficients
 \begin{subequations} \label{eqn:matcgmpt}
\begin{align}
( {\mathbf C}_{kj}^{{\rm C},\ell t})_{pq}  :=&-
 \im \nu   {\vec e}_k \cdot  
\int_B {\vec \xi} \times \left (  \overline{ H_\ell^p ({\vec \xi}) } ({\vec \phi}_j^{(0),t,q} + {\vec \phi}_j^{(1),t,q } )\right ) \dif {\vec \xi},  \\
( {\mathbf N}_{kj}^{{\rm C},\ell t})_{pq} := &  \left ( 1 - \mu_r^{-1} \right )   {\vec  e}_k \cdot  \int_B \overline{ H_\ell^p({\vec \xi}) }  \left ( \nabla_\xi \times {\vec \phi}_{j }^{(1),t,q}   \right ) \dif {\vec \xi} \nonumber \\ 
 & + \left ( 1 - \mu_r^{-1}  \right )  {\vec  e}_k \cdot  \int_B \overline{H_\ell^p ({\vec \xi})  } \left (  \nabla_\xi \times {\vec \phi }_{j }^{(0), t,q}   \right ) \dif {\vec \xi},  
\end{align}
\end{subequations}
which are of dimension $(2\ell+1)\times(2t+1)$,
so that
\begin{align*}
{ \mathfrak M}_{kj}^{{\rm C}, \ell p t q } = -\frac{ \alpha^{3+\ell +t}(-1)^{\ell }}{  2(\ell +1) (t+2) (2\ell+1)(2t+1)} ( {\mathbf C}_{kj}^{{\rm C},\ell t})_{pq} + \frac{ \alpha^{3+\ell +t} (-1)^{\ell }}{(2\ell+1)(2t+1)}   \frac{1}{t+2}   ( {\mathbf N}_{kj}^{{\rm C},\ell t})_{pq} ,
\end{align*} 
for $ -\ell \le p \le \ell , -t \le q \le t $,
%
and the 
$1\times (2t+1)$ and $1 \times (2\ell+1)$ matrices ${\mathbf D}{\mathbf K}_{{\mathrm{s}j}}^{t}$ and ${\mathbf D}{\mathbf K}_{{\mathrm{r}k}}^{\ell}$  with coefficients
\begin{align}
({\mathbf D}{\mathbf K}_{\mathrm{s} j}^{t})_n  & = \sum_{o=1}^3 \left ( {\vec D}_x^2  ( K_{t}^n ( {\vec x}  )) |_{{\vec x}= {\vec x}^{{\rm{s}}}}  \right )_{jo} \nonumber
  d_o,  \qquad -t \le n \le t , \\
( {\mathbf D}{\mathbf K}_{{\mathrm{r}k}}^{\ell})_m  & = \sum_{i=1}^3   \left ({\vec D}_{ x}^{2} ( K_{\ell}^p ({\vec x}))  |_{{\vec x}=  {\vec x}^{{\rm{r}}} }  \right )_{ik}  {f}_i ,
\qquad  -\ell \le m \le \ell .
\nonumber
\end{align}
Thus,
\begin{align*}
V_{{\rm{s} \rm{r}}}  = &
- \sum_{\ell =0}^{M-1} \sum_{t =0}^{M-1-\ell} \frac{ \alpha^{3+\ell +t}(-1)^{\ell }}{  2(\ell +1) (t+2) (2\ell+1)(2t+1)}    \sum_{k,j=1}^3  {\mathbf D}{\mathbf K}_{{\mathrm{r}k}}^{\ell}  {\mathbf C}_{kj}^{{\rm C},\ell t} ( {\mathbf D}{\mathbf K}_{\mathrm{s} j}^{t} )^T \nonumber \\
&+\sum_{\ell =0}^{M-1} \sum_{t =0}^{M-1-\ell}  \frac{ \alpha^{3+\ell +t} (-1)^{\ell }}{(2\ell+1)(2t+1)}   \frac{1}{t+2}     \sum_{k,j=1}^3  {\mathbf D}{\mathbf K}_{{\mathrm{r}k}}^{\ell}  {\mathbf N }_{kj}^{{\rm C}, \ell t} ( {\mathbf D}{\mathbf K}_{\mathrm{s} j}^{t} )^T ,
\end{align*}
is an alternative form to (\ref{eqn:indvoltharm}).
\subsection{Transformation formulae for HGMPTs}

We present results for the scaling, shifting and rotation of the HGMPT matrices. It is useful to introduce the $(p+1) \times (p+1) $ matrices ${\mathbf A}_p^{\IH}$ with entries
\begin{align*}
({\mathbf A}_p^{\IH} )_{mn} : = a_{nm}^{\IH}, \qquad -p \le n \le p, -p \le m \le p,
\end{align*}
which is unitary if $H_n^m$  and $I_n^m$ are chosen appropriately~\cite{ledger2022}, so that we can write
\begin{align*}
 {\mathbf C}_{kj}^{{\rm C},\ell t} =  {\mathbf A}_\ell^{\IH}  {\mathbf C}_{kj}^{{\rm H}, \ell t} (  {\mathbf A}_t^{\IH} )^* , \qquad  {\mathbf N}_{kj}^{{\rm C},\ell t} =  {\mathbf A}_\ell^{\IH}  {\mathbf N}_{kj}^{{\rm H}, \ell t} (  {\mathbf A}_t^{\IH} )^*,
\end{align*}
where $*$ denotes the complex conjugate transpose and
\begin{align*}
 {\mathbf C}_{kj}^{{\rm H}, \ell t} =  ({\mathbf A}_\ell^{\IH} )^* {\mathbf C}_{kj}^{{\rm C}, \ell t}   {\mathbf A}_t^{\IH} , \qquad  {\mathbf N}_{kj}^{{\rm H}, \ell r} =  ({\mathbf A}_\ell^{\IH} )^* {\mathbf N}_{kj}^{{\rm C}, \ell t}   {\mathbf A}_t^{\IH} .
\end{align*}

\subsubsection{Scaling}

\begin{lemma}
For any positive integers $\ell,t $  in the following  and a real scaling parameter $s > 0$ the following holds
\begin{align}
 {\mathbf C}_{kj}^{{\rm C},\ell t}[s\alpha B, \nu,\mu_r]  = &  s^3 {\mathbf C}_{kj}^{{\rm C}, \ell t} [\alpha B, s^2 \nu, \mu_r],  \nonumber \\
 {\mathbf N}_{kj}^{{\rm C}, \ell t} [s\alpha B, \nu,\mu_r]= & s^3  {\mathbf N}_{kj}^{{\rm C,}\ell t} [\alpha B, s^2 \nu, \mu_r] ,\nonumber
\end{align} 
where $ [s\alpha B, \nu,\mu_r]$ indicates evaluation for an object $s\alpha B$ with material parameters $\nu$ and $\mu_r$.
\end{lemma}
\begin{proof}
Let ${\vec \phi}_{j,B}^{(0),t,n}({\vec \xi}')$ be the solution to  (\ref{eqn:harmonicphi0}). Then, since $H_t^n(s {\vec \xi}') = s^t H_t^n({\vec \xi }')$, we find that
\begin{align*}
\frac{1}{s^{1+t}} {\vec \phi}_{j,sB}^{(0),t,n} (s {\vec \xi}') = {\vec \phi}_{j,B}^{(0),t,n}( {\vec \xi}'),
\end{align*}
where ${\vec \phi}_{j,sB}^{(0),t,n} $ is the solution to (\ref{eqn:harmonicphi0}) with $B$ replaced by $sB$. If  ${\vec \phi}_{j,B}^{(1),t,n}[s^2 \nu] $ is the solution to (\ref{eqn:harmonicphi1}) with $\nu$ replaced by $s^2\nu$, we find that
\begin{align*}
\frac{1}{s^{1+t}} {\vec \phi}_{j,sB}^{(1),t,n} [\nu]  (s {\vec \xi}') = {\vec \phi}_{j,B}^{(1),t,n} [s^2 \nu] ( {\vec \xi}') ,
\end{align*}
where ${\vec \phi}_{j,sB}^{(1),t,n} [ \nu] $ is the solution to  (\ref{eqn:harmonicphi1})  with $B$ replaced  by $sB$. Then, by the application of these results, we find that
\begin{align}
  {\mathfrak M}_{kj}^{{\rm C},\ell m t n } [s\alpha B, \nu , \mu_r]  =& \frac{1}{(2\ell +1)(2t +1)}    \left (   \frac{\im \nu \alpha^{3+\ell +t}(-1)^{\ell }}{ 2 (\ell +1) (t +2)} {\vec e}_k \cdot  
\int_{sB} {\vec \xi} \times \left (  \overline{ H_\ell^m({\vec \xi})}  ({\vec \phi}_{j,sB}^{(0),t,n} + {\vec \phi}_{j,sB}^{(1),t,n} [\nu] )\right ) \dif {\vec \xi}    \right . \nonumber \\
 & + \left ( 1 - \mu_r^{-1}  \right )  \alpha^{3+\ell+t} (-1)^{\ell } {\vec  e}_k \cdot  \int_{sB} \overline{ H_\ell^m( {\vec \xi})  } \left ( \frac{1}{t+2} \nabla_\xi \times {\vec \phi}_{j,sB }^{(1),t,n} [\nu]  \right ) \dif {\vec \xi} , \nonumber \\ 
 & \left . +  \left ( 1 - \mu_r^{-1} \right )  \alpha^{3+\ell+t } (-1)^{\ell } {\vec  e}_k \cdot  \int_{sB} \overline{H_\ell^m (  {\vec \xi}) }  \left ( \frac{1}{t+2} \nabla_\xi \times {\vec \phi}_{j ,sB}^{(0),t,n}   \right ) \dif {\vec \xi} \right )   \nonumber \\
= & \frac{1}{(2\ell +1)(2t  +1)} s^3   \left (   \frac{\im \nu \alpha^{3+\ell +t }(-1)^{\ell }}{ 2 (\ell +1) (t +2)} {\vec e}_k \cdot  \right . \nonumber\\
&\int_{B} s{\vec \xi}' \times \left (  \overline{ H_\ell^m(s {\vec \xi}')}   ({\vec \phi}_{j,sB}^{(0),t,n}(s{\vec \xi}') + {\vec \phi}_{j,sB}^{(1),t,n}[\nu] (s{\vec \xi}') )\right ) \dif {\vec \xi} '    \nonumber \\
 & + \left ( 1 - \mu_r^{-1}  \right )  \alpha^{3+\ell+t} (-1)^{\ell } {\vec  e}_k \cdot  \int_{B} \overline{ H_\ell^m (s{\vec \xi}')  } \left ( \frac{1}{t+2} \frac{1}{s} \nabla_{\xi'} \times ( s^{1+t} {\vec \phi}_{j,B }^{(1),t,n} [s^2\nu])   \right ) \dif {\vec \xi}' , \nonumber \\ 
 & \left . +  \left ( 1 - \mu_r^{-1} \right )  \alpha^{3+\ell+t } (-1)^{\ell } {\vec  e}_k \cdot  \int_{B} \overline{ H_\ell^m ( s {\vec \xi}')}   \left ( \frac{1}{t+2} \frac{1}{s} \nabla_{\xi'} \times (s^{1+ t}  {\vec \phi}_{j,B }^{(0),t,n})   \right ) \dif {\vec \xi}' \right )   \nonumber \\
 = & \frac{1}{(2\ell +1)(2t +1)} s^{3+\ell +t}    \left (   \frac{\im (s^2\nu) \alpha^{3+\ell +t}(-1)^{\ell }}{ 2 (\ell +1) (t +2)} {\vec e}_k \cdot  \right . \nonumber \\
&\int_{B} {\vec \xi}' \times \left (  \overline{ H_\ell^m( {\vec \xi}')}  ({\vec \phi}_{j,B}^{(0),t,n}({\vec \xi}') + {\vec \phi}_{j,B}^{(1),t,n}[s^2 \nu] ({\vec \xi}') )\right ) \dif {\vec \xi} '    \nonumber \\
 & + \left ( 1 - \mu_r^{-1}  \right )  \alpha^{3+\ell+t} (-1)^{\ell } {\vec  e}_k \cdot  \int_{B} \overline{H_\ell^m ( {\vec \xi}')}   \left ( \frac{1}{t+2}  \nabla_{\xi'} \times (  {\vec \phi}_{j,B }^{(1),t,n} [s^2\nu])   \right ) \dif {\vec \xi}' , \nonumber \\ 
 & \left . +  \left ( 1 - \mu_r^{-1} \right )  \alpha^{3+\ell+t } (-1)^{\ell } {\vec  e}_k \cdot  \int_{B} \overline{H_\ell^m (  {\vec \xi}')}   \left ( \frac{1}{t+2}  \nabla_{\xi'} \times (   {\vec \phi}_{j,B }^{(0),t,n})   \right ) \dif {\vec \xi}' \right )   \nonumber \\
 = &  s^{3+\ell +t} {\mathfrak M}_{kj}^{\ell m t n } [\alpha B, s^2 \nu , \mu_r], \nonumber
\end{align}
which, by replacing $m$ with $p$ and $n$ by $q$, and using the definitions
 of $ ( {\mathbf C}_{kj}^{{\rm C}, \ell t})_{pq} $ and $ ( {\mathbf N}_{kj}^{{\rm C}, \ell t})_{pq} $ in (\ref{eqn:matcgmpt}), completes the proof.
\end{proof}

\subsection{Translation}
To deal with a translation (shift) of a HGMPT, we first recall the translation of $H_n^m({\vec \xi}) = r^n T_n^m(\theta, \varphi)$ where ${\vec \xi}$ has spherical coordinates $(r,\theta,\varphi) $. For ${\vec z}$ with spherical coordinates $(r_z,\theta_z,\varphi_z)$ and ${\vec \xi}'$  with spherical coordinates $ {\vec \xi} + {\vec z} = (r', \theta',\varphi')$, Ammari {\it et al.}~\cite{ammari3dinvgpt} provide the following 
\begin{align*}
H_n^m({\vec \xi}')  = {r'}^{n} Y_n^m(\theta',\varphi') = &  \sum_{(\nu,\mu)}^{(n,m)} C_{\nu \mu nm} r_z^{n-\mu} Y_{n-\nu}^{m-\mu} (\theta_z , \varphi_z)r ^\nu Y_\nu^\mu (\theta, \varphi) \nonumber \\
= &  \sum_{(\nu,\mu)}^{(n,m)} C_{\nu \mu nm}  H_{n-\nu}^{m-\mu} ({\vec z})H_\nu^\mu ({\vec \xi}),
\end{align*}
for the translation of a spherical harmonic, which we have chosen to write in terms of $H_n^m(\cdot)$. In the above, the real coefficient $C_{\nu \mu n m}$ and the special form of summation are as defined in~\cite{ammari3dinvgpt}. 

Building on the translation invariance property of the rank 2 MPT established in Proposition 5.1 of~\cite{Ammari2015},  and the translation properties of HGPTs in Lemma 4.2 of~\cite{ledger2022}, we establish the following for the translation of HGMPTs.

\begin{lemma}
For any positive integers $\ell,t $  in the following  and a translation of $B$ to $B_z$ by a constant vector ${\vec z}$ we have
\begin{align} 
( {\mathbf C}_{kj}^{{\rm C},\ell t} [B_z] )_{mn } = & \sum_{(\nu,\mu)}^{(\ell,m)} C_{\nu \mu \ell m}  \overline{ H_{\ell -\nu}^{m-\mu} ({\vec z})} \sum_{(\tau,\lambda )}^{(t,n)} C_{\tau \lambda t n}  \overline{ H_{t -\tau}^{n-\lambda } ({\vec z}) }
( {\mathbf C}_{kj}^{{\rm C},\nu \mu }[B] )_{\mu \tau }  \nonumber ,\\
( {\mathbf N}_{kj}^{{\rm C},\ell t} [B_z] )_{mn } = & \sum_{(\nu,\mu)}^{(\ell,m)} C_{\nu \mu \ell m}  \overline{ H_{\ell -\nu}^{m-\mu} ({\vec z})} \sum_{(\tau,\lambda )}^{(t,n)} C_{\tau \lambda t n}  \overline{ H_{t -\tau}^{n-\lambda } ({\vec z}) }
( {\mathbf N}_{kj}^{{\rm C},\nu \mu  } [B] )_{\mu \tau } \nonumber .
\end{align}
\end{lemma}
\begin{proof}
Let ${\vec F}_z^{(0),t,n}$ satisfy 
\begin{subequations}  
\begin{align}
\nabla_{\xi'} \times \mu_r^{-1} \nabla_{\xi'} \times {\vec F}_z^{(0),t,n} & = {\vec 0}  && \text{in $B_z$}, \\
\nabla_{\xi'} \cdot {\vec F}_z^{(0),t,n}  = 0 , \qquad \nabla_\xi \times  \nabla_\xi \times {\vec F}_z^{(0),t,n}   & = {\vec 0} && \text{in ${\mathbb R}^3 \setminus \overline{B_z}$} , \\
[{\vec n} \times {\vec F}_z^{(0),t,n}  ]_\Gamma &  = {\vec 0} && \text{on $\partial B_z$},\\
  [{\vec n} \times \tilde{\mu}_r^{-1} \nabla_{\xi'} \times {\vec F}_z^{(0),t,n} ]_\Gamma & ={\vec 0}  && \text{on $\partial B_z$},\\
{\vec F}_z^{(0),t,n} - \overline{ H_t^n( {\vec \xi}')} {\vec e}_j \times {\vec \xi}'
 & = O( | {\vec \xi}' |^{-1}) && \text{as $|{\vec \xi}' | \to \infty$},
\end{align}
\end{subequations}
and ${\vec F}_z^{(1),t,n}$ satisify
\begin{subequations} \label{eqn:f1zproblem}
\begin{align}
\nabla_{\xi'} \times \mu_r^{-1} \nabla_{\xi'} \times{\vec F}_z^{(1),t,n} - \im \nu ( {\vec F}_z^{(1),t,n}  + {\vec F}_z^{(0),t,n} ) & = {\vec 0}  && \text{in $B_z$}, \\
\nabla_\xi \cdot{\vec F}_z^{(1),t,n} = 0 , \qquad \nabla_{\xi'} \times  \nabla_{\xi'} \times{\vec F}_z^{(1),t,n} & = {\vec 0} && \text{in ${\mathbb R}^3 \setminus \overline{B_z}$} , \\
[{\vec n} \times{\vec F}_z^{(1),t,n} ]_\Gamma &  = {\vec 0} && \text{on $\partial B_z$},\\
  [{\vec n} \times \tilde{\mu}_r^{-1} \nabla_\xi \times{\vec F}_z^{(1),t,n}  ]_\Gamma & ={\vec 0}  && \text{on $\partial B_z $},\\
\int_{\partial B_z}  {\vec n} \cdot{\vec F}_z^{(1),t,n} \dif {\vec \xi}' & = 0 ,\\
{\vec F}_z^{(1),t,n} & = O( | {\vec \xi}' |^{-1}) && \text{as $|{\vec \xi}' | \to \infty$},
\end{align}
\end{subequations}
and ${\vec F}_0^{(0),t,n}$, ${\vec F}_0^{(1),t,n}$ be the corresponding solutions for ${\vec z}={\vec 0}$.
Then, 
\begin{align*}
 \overline{ H_t^n( {\vec \xi}') } {\vec e}_j \times {\vec \xi}' =  \sum_{(\nu,\mu)}^{(t,n)} C_{\nu \mu t n }  \overline{ H_{t-\nu}^{n-\mu} ({\vec z}) } \overline{ H_\nu^\mu ({\vec \xi}) }{\vec e}_j \times ( {\vec \xi} + {\vec z}) ,
\end{align*} 
since $C_{\nu \mu rn } $ is real, we have
\begin{subequations} \label{eqn:fintermsofg}
\begin{align}
{\vec F}_z^{(0),t,n} & =  \sum_{(\nu,\mu)}^{(t,n)} C_{\nu \mu t n }  \overline {H_{t -\nu}^{n-\mu} ({\vec z})}  {\vec F}_0^{(0),\nu ,\mu} + \sum_{(\nu,\mu)}^{(t ,n)} C_{\nu \mu tn } \overline{  H_{t -\nu}^{n-\mu} ({\vec z})}
{\vec G}_0^{(0),\nu ,\mu}   , \\
{\vec F}_z^{(1),t ,n} & =  \sum_{(\nu,\mu)}^{(t,n)} C_{\nu \mu t n }  \overline{ H_{t -\nu}^{n-\mu} ({\vec z}) } {\vec F}_0^{(1),\nu ,\mu} + \sum_{(\nu,\mu)}^{(t,n)} C_{\nu \mu t n }  \overline{ H_{t-\nu}^{n-\mu} ({\vec z})}
{\vec G}_0^{(1),\nu ,\mu}  .
\end{align}
\end{subequations}
In the above, ${\vec G}_0^{(0),t,n}$ satisfies
\begin{subequations}  
\begin{align}
\nabla_{\xi} \times \mu_r^{-1} \nabla_{\xi} \times {\vec G}_0^{(0),t,n} & = {\vec 0}  && \text{in $B$}, \\
\nabla_{\xi} \cdot {\vec G}_0^{(0),t,n}  = 0 , \qquad \nabla_\xi \times  \nabla_\xi \times {\vec G}_0^{(0),t,n}   & = {\vec 0} && \text{in $B^c$} , \\
[{\vec n} \times {\vec G}_0^{(0),t,n}  ]_\Gamma &  = {\vec 0} && \text{on $\partial B$},\\
  [{\vec n} \times \tilde{\mu}_r^{-1} \nabla_{\xi} \times {\vec G}_0^{(0),t,n} ]_\Gamma & ={\vec 0}  && \text{on $\partial B$},\\
{\vec G}_0^{(0),t,n} -  \overline{ H_t^n( {\vec \xi}) } {\vec e}_j \times {\vec z}
 & = O( | {\vec \xi} |^{-1}) && \text{as $|{\vec \xi} | \to \infty$},
\end{align}
\end{subequations}
and ${\vec G}_0^{(1),t,n}$ satisfies
\begin{subequations} 
\begin{align}
\nabla_{\xi} \times \mu_r^{-1} \nabla_{\xi} \times{\vec G}_0^{(1),t,n} - \im \nu ( {\vec G}_0^{(1),t,n}  + {\vec G}_0^{(0),t,n} ) & = {\vec 0}  && \text{in $B$}, \\
\nabla_\xi \cdot{\vec G}_0^{(1),t,n} = 0 , \qquad \nabla_{\xi} \times  \nabla_{\xi} \times{\vec G}_0^{(1),t,n} & = {\vec 0} && \text{in $B^c$} , \\
[{\vec n} \times{\vec G}_0^{(1),t,n} ]_\Gamma &  = {\vec 0} && \text{on $\partial B$},\\
  [{\vec n} \times \tilde{\mu}_r^{-1} \nabla_\xi \times{\vec G}_0^{(1),t,n}  ]_\Gamma & ={\vec 0}  && \text{on $\partial B $},\\
\int_{\partial B}  {\vec n} \cdot{\vec G}_0^{(1),t,n} \dif {\vec \xi} & = 0 ,\\
{\vec G}_0^{(1),t,n} & = O( | {\vec \xi}' |^{-1}) && \text{as $|{\vec \xi} | \to \infty$}.
\end{align}
\end{subequations}

Setting $ {\vec G}_0^{(0),t,n}  = \overline{  H_t^n( {\vec \xi})} {\vec e}_j \times {\vec z} =\nabla u$ in $B$ then we can define $\tilde{u}$ to be the solution to
\begin{subequations}
\begin{align}
\nabla^2 \tilde{u} & = 0 && \text{in $B^c$}, \\ 
\tilde{u} & = u && \text{on $\Gamma$}, \\ 
\tilde{u} & = O(|{\vec \xi}|^{-1}) && \text{as $|{\vec \xi}| \to \infty$},
\end{align}
\end{subequations}
we have 
\begin{equation*}
{\vec G}_0^{(0),t,n} = \left \{ \begin{array}{ll} \overline{ H_t^n( {\vec \xi}) } {\vec e}_j \times {\vec z}  & \text{in $B$} \\
\nabla{ \tilde{u}}  & \text{in $B^c$} \end{array} \right .  , \qquad {\vec G}_0^{(1),t,n} = {\vec 0}.
\end{equation*}
This means that
\begin{align}
  {\mathfrak M}_{kj}^{{\rm C},\ell m t n }[B_z]  = &  
   \frac{\im \nu \alpha^{3+\ell +t}(-1)^{\ell }}{ 2 (\ell +1) (t+2) (2\ell +1)(2t+1)} {\vec e}_k \cdot  
\int_{B_z}  {\vec \xi}' \times \left (  \overline{ H_\ell^m ({\vec \xi}') } (  {\vec F}_z^{(0),t ,n}   +  {\vec F}_z^{(1),t,n}) \right ) \dif {\vec \xi}'   \nonumber \\
 & +  \left ( 1 - \mu_r^{-1} \right ) \frac{ \alpha^{3+\ell +t} (-1)^{\ell }}{(2\ell+1)(2t+1)} {\vec  e}_k \cdot  \int_{B_z}  \overline{H_\ell^m ({\vec \xi}') } \left ( \frac{1}{t+2} \nabla_{\xi'} \times {\vec F}_z^{(1),t  , n }   \right ) \dif {\vec \xi}'  \nonumber \\ 
 &   +  \left ( 1 - \mu_r^{-1}  \right )  \frac{\alpha^{3+\ell +t} (-1)^{\ell }}{(2\ell+1)(2t+1)} {\vec  e}_k \cdot  \int_{B_z} \overline{ H_\ell^m ({\vec \xi}')}   \left ( \frac{1}{t+2} \nabla_{\xi'} \times  {\vec F}_z^{(0),t , n}    \right ) \dif {\vec \xi}' \nonumber  \\
=  & \frac{ \alpha^{3+\ell +t}(-1)^{\ell }}{ 2 (\ell +1) (t+2) (2\ell +1)(2t+1)} {\vec e}_k \cdot  
\int_{B_z}  {\vec \xi}' \times \left (  \overline{H_\ell^m ({\vec \xi}')}  \nabla_{\xi'} \times \nabla_{\xi'} \times       {\vec F}_z^{(1),t,n} \right ) \dif {\vec \xi}'   \nonumber \\
 & +  \left ( 1 - \mu_r^{-1} \right ) \frac{ \alpha^{3+\ell +t} (-1)^{\ell }}{(2\ell+1)(2t+1)} {\vec  e}_k \cdot  \int_{B_z}  \overline{H_\ell^m ({\vec \xi}')}  \left ( \frac{1}{t+2} \nabla_{\xi'} \times {\vec F}_z^{(1),t  , n }   \right ) \dif {\vec \xi}'  \nonumber \\ 
 &   +  \left ( 1 - \mu_r^{-1}  \right )  \frac{\alpha^{3+\ell +t} (-1)^{\ell }}{(2\ell+1)(2t+1)} {\vec  e}_k \cdot  \int_{B_z} \overline{H_\ell^m ({\vec \xi}')}   \left ( \frac{1}{t+2} \nabla_{\xi'} \times  {\vec F}_z^{(0),t , n}    \right ) \dif {\vec \xi}' \nonumber   ,
 \end{align}
 by using the transmission problem (\ref{eqn:f1zproblem}). Next, using (\ref{eqn:fintermsofg}) and (\ref{eqn:matcgmpt}), we get
\begin{align} 
( {\mathbf C}_{kj}^{{\rm C},\ell t} [B_z] )_{mn }  =& -  \sum_{(\nu,\mu)}^{(\ell,m)} C_{\nu \mu \ell m}  \overline{H_{\ell -\nu}^{m-\mu} ({\vec z}) }
  \sum_{(\tau,\lambda )}^{(t,n)} C_{\tau \lambda t n}  \overline{ H_{t -\tau}^{n-\lambda } ({\vec z}) }\nonumber \\
 &    {\vec e}_k \cdot  
\int_B ( {\vec \xi} + {\vec z})  \times \left (  \overline{H_\nu^\mu ({\vec \xi}) } \nabla_{\xi} \times \nabla_{\xi} \times       {\vec F}_0^{(1),\tau ,\lambda } \right )  \dif {\vec \xi}   \nonumber \\
( {\mathbf N}_{kj}^{{\rm C},\ell t} [B_z] )_{mn }  = & \sum_{(\nu,\mu)}^{(\ell,m)} C_{\nu \mu \ell m}  \overline{H_{\ell -\nu}^{m-\mu} ({\vec z}) }
  \sum_{(\tau,\lambda )}^{(t,n)} C_{\tau \lambda t n}  \overline{ H_{t -\tau}^{n-\lambda } ({\vec z}) }\nonumber \\
&\left (  \left ( 1 - \mu_r^{-1} \right )  {\vec  e}_k \cdot  \int_B  \overline{H_\nu^\mu ({\vec \xi})}  \left ( \nabla_\xi \times {\vec F}_0^{(1),\tau ,\lambda}   \right ) \dif {\vec \xi} \right .  \nonumber \\ 
 &   +  \left ( 1 - \mu_r^{-1}  \right )  {\vec  e}_k \cdot  \int_B \overline{H_\nu^\mu ({\vec \xi}) }  \left ( \nabla_\xi \times  {\vec F}_0^{(0),\tau ,\lambda}    \right ) \dif {\vec \xi} \nonumber \\
  &  \left . +  \left ( 1 - \mu_r^{-1}  \right )   {\vec  e}_k \cdot  \int_B \overline{H_\nu^\mu ({\vec \xi})}   \left (  \nabla_\xi \times  {\vec G}_0^{(0),\tau ,\lambda}    \right )  \dif {\vec \xi} \right ), \nonumber
   \end{align}
with our final result immediately following, since, by replacing $j$ with $k$ in ${\vec G}_0^{(0), \nu,\mu}$, and recalling $\nabla \times{\vec G}_0^{(0), \nu,\mu} = {\vec 0}$,  then
\begin{align}
{\vec e}_k&  \cdot \int_B {\vec z}  \times \left (  \overline{ H_\nu^\mu ({\vec \xi}) }  \nabla_{\xi} \times \nabla_{\xi} \times       {\vec F}_0^{(1),t,n} \right )  \dif {\vec \xi} 
= \int_B  \nabla_{\xi} \times \nabla_{\xi} \times       {\vec F}_0^{(1),t,n} \cdot {{\vec G}_0^{(0), \nu,\mu}} \dif {\vec \xi} = { 0}, \nonumber
\end{align}
by performing integration by parts.
\end{proof}

\subsection{Rotation} \label{sect:rot}
Consider
a general rotation matrix ${\mathbf R}$ in terms of the Euler angles $\gamma,\beta, \alpha$ for rotations about the $x_1$, $x_2$ and $x_3$ axes, respectively, as
\begin{align*}
{\mathbf R}= \left ( 
\begin{array}{ccc}
\cos \gamma & - \sin \gamma & 0 \\
\sin \gamma & \cos \gamma & 0 \\
0 & 0 & 1 \end{array} \right )
\left ( \begin{array}{ccc} 
\cos \beta & 0 & - \sin \beta \\
0 & 1 & 0 \\
\sin \beta & 0 & \cos \beta \\
\end{array} \right )
\left ( \begin{array}{ccc}
\cos \alpha & - \sin \alpha & 0 \\
\sin \alpha & \cos \alpha & 0 \\
0 & 0 & 1\end{array} 
\right ).
\end{align*}
Following (4.8) ~\cite{ammari3dinvgpt} we have
\begin{align}
H_n^m({\mathbf R}{\vec \xi}) = \sum_{m'=-n}^n \rho_n^{m',m} (\alpha, \beta, \gamma) H_n^{m'} ({\vec \xi}) \label{eqn:rotsphm},
\end{align}
for where $\rho_n^{m',m}$ is as defined in~\cite{ammari3dinvgpt}. 

In the following, we extend the results for the transformation of CGPTs under the action of ${\mathbf R}$ obtained by Ammari {it et al.} in their Lemma 3.2~\cite{ammari3dinvgpt} and the transformation of MPTs under ${\mathbf R}$ obtained in Theorem 3.1 of~\cite{LedgerLionheart2015} to the transformation of HGMPTs.

\begin{lemma}
For any positive integers $\ell,r $  in the following  and a transformation of $B$ to ${\mathbf R}(B)$ by an orthogonal rotation matrix ${\mathbf R}$ 
\begin{align}
 {\mathbf C}_{kj}^{{\rm C},\ell t}  [{\mathbf R} (B) ] & =({\mathbf R})_{ku} ({\mathbf R})_{jv}  ( \overline{{\mathbf Q}_\ell({\mathbf R}})    {\mathbf C}_{uv}^{{\rm C},\ell t}  [ B ] 
   \overline{{\mathbf Q}_t ({\mathbf R})}^T) ,  \nonumber \\
 {\mathbf N}_{kj}^{{\rm C}, \ell t}  [{\mathbf R} (B) ] & = ({\mathbf R})_{ku} ({\mathbf R})_{jv}  ( \overline{{\mathbf Q}_\ell({\mathbf R}})    {\mathbf N}_{uv}^{{\rm C}, \ell t}  [ B ] 
   \overline{{\mathbf Q}_t ({\mathbf R})}^T) , \nonumber
\end{align}
where 
\begin{align}
{\mathbf Q}_\ell  ({\mathbf R}) := \left ( \begin{array}{cccc}
\rho_\ell^{-\ell,-\ell} & \rho_\ell^{-\ell+1,-\ell} & \cdots & \rho_\ell^{\ell,-\ell}  \\
\rho_\ell^{-\ell,-\ell+1} & \rho_\ell^{-\ell+1,-\ell+1} & \cdots & \rho_\ell^{\ell,-\ell+1}  \\
\ldots & \ldots  & \ddots & \ldots \\
\rho_\ell^{-\ell,\ell} & \rho_\ell^{-\ell+1,\ell} & \cdots & \rho_\ell^{\ell,\ell}  \end{array}
 \right ) \label{eqn:defineq} .
 \end{align}
\end{lemma}
\begin{proof}
Let ${\vec F}_{B,{\vec e}_j}^{(0),t,n}$  satisfy
\begin{subequations}  
\begin{align}
\nabla_{\xi} \times \mu_r^{-1} \nabla_{\xi} \times {\vec F}_{B,{\vec e}_j}^{(0),t,n}  & = {\vec 0}  && \text{in $B$}, \\
\nabla_{\xi} \cdot {\vec F}_{B,{\vec e}_j}^{(0),t,n}   = 0 , \qquad \nabla_\xi \times  \nabla_\xi \times {\vec F}_{B,{\vec e}_j}^{(0),t,n}   & = {\vec 0} && \text{in $ B^c$} , \\
[{\vec n} \times {\vec F}_{B,{\vec e}_j}^{(0),t,n}   ]_\Gamma &  = {\vec 0} && \text{on $\partial B$},\\
  [{\vec n} \times \tilde{\mu}_r^{-1} \nabla_{\xi} \times {\vec F}_{B,{\vec e}_j}^{(0),t,n}  ]_\Gamma & ={\vec 0}  && \text{on $\partial B$},\\
{\vec F}_{B,{\vec e}_j}^{(0),t,n} - \overline{  H_t^n( {\vec \xi}) } {\vec e}_j \times {\vec \xi}
 & = O( | {\vec \xi} |^{-1}) && \text{as $|{\vec \xi} | \to \infty$},
\end{align}
\end{subequations}
and ${\vec F}_{B,{\vec e}_j}^{(1),t,n}$ satisfy
\begin{subequations} \label{eqn:f1problem}
\begin{align}
\nabla_{\xi'} \times \mu_r^{-1} \nabla_{\xi} \times {\vec F}_{B,{\vec e}_j}^{(1),t,n} - \im \nu ({\vec F}_{B,{\vec e}_j}^{(0),t,n}   + {\vec F}_{B,{\vec e}_j}^{(1),t,n}  ) & = {\vec 0}  && \text{in $B$}, \\
\nabla_\xi \cdot {\vec F}_{B,{\vec e}_j}^{(1),t,n}  = 0 , \qquad \nabla_{\xi} \times  \nabla_{\xi} \times {\vec F}_{B,{\vec e}_j}^{(1),t,n}  & = {\vec 0} && \text{in $B^c$} , \\
[{\vec n} \times {\vec F}_{B,{\vec e}_j}^{(1),t,n} ]_\Gamma &  = {\vec 0} && \text{on $\partial B$},\\
  [{\vec n} \times \tilde{\mu}_r^{-1} \nabla_\xi \times {\vec F}_{B,{\vec e}_j}^{(1),t,n}  ]_\Gamma & ={\vec 0}  && \text{on $\partial B $},\\
\int_{\partial B}  {\vec n} \cdot {\vec F}_{B,{\vec e}_j}^{(1),t,n} \dif {\vec \xi} & = 0 ,\\
{\vec F}_{B,{\vec e}_j}^{(1),t,n}  & = O( | {\vec \xi} |^{-1}) && \text{as $|{\vec \xi} | \to \infty$}.
\end{align}
\end{subequations}
then, by combining (\ref{eqn:rotsphm}) and Proposition 5.3 of~\cite{Ammari2015}, we have
\begin{align}
{\vec F}_{{\mathbf R} (B),{\vec e}_j}^{(0),t,n} ({\mathbf R}{\vec \xi}) =  \sum_{n'=-t}^t \overline{\rho_t^{n',n} (\alpha, \beta, \gamma)}{\mathbf R} {\vec F}_{B,{\mathbf R}^T {\vec e}_j}^{(0),t,n'} ,\qquad
{\vec F}_{{\mathbf R} (B),{\vec e}_j}^{(1),t,n} ({\mathbf R}{\vec \xi}) =  \sum_{n'=-t}^t \overline{\rho_t^{n',n} (\alpha, \beta, \gamma){\mathbf R}} {\vec F}_{B,{\mathbf R}^T {\vec e}_j}^{(1),t,n'} ,\nonumber
\end{align}
for all ${\vec \xi} \in {\mathbb R}^3$. Hence,
\begin{align}
  {\mathfrak M}_{kj}^{{\rm C},\ell m t n }[{\mathbf R}( B)]  = &  
   \frac{\im \nu \alpha^{3+\ell +t}(-1)^{\ell }}{ 2 (\ell +1) (t+2) (2\ell +1)(2t+1)} {\vec e}_k \cdot  
\int_{{\mathbf R}(B)}  {\vec \xi} \times \left (  \overline{ H_\ell^m ({\mathbf R} {\vec \xi} ) } (  {\vec F}_{{\mathbf R} (B),{\vec e}_j}^{(0),t,n}     + {\vec F}_{{\mathbf R} (B),{\vec e}_j}^{(1),t,n})  \right ) \dif {\vec \xi}  \nonumber \\
 & +  \left ( 1 - \mu_r^{-1} \right ) \frac{ \alpha^{3+\ell +t} (-1)^{\ell }}{(2\ell+1)(2t+1)} {\vec  e}_k \cdot  \int_{{\mathbf R}(B)}  \overline{H_\ell^m ( {\mathbf R} {\vec \xi} ) } \left ( \frac{1}{t+2} \nabla  \times {\vec F}_{{\mathbf R} (B),{\vec e}_j}^{(1),t,n}    \right ) \dif {\vec \xi}  \nonumber \\ 
 &   +  \left ( 1 - \mu_r^{-1}  \right )  \frac{\alpha^{3+\ell +t} (-1)^{\ell }}{(2\ell+1)(2t+1)} {\vec  e}_k \cdot  \int_{{\mathbf R}( B) } \overline{ H_\ell^m ( {\mathbf R}( {\vec \xi} ) )}   \left ( \frac{1}{t+2} \nabla_{\xi} \times {\vec F}_{{\mathbf R} (B),{\vec e}_j}^{(0),t,n}     \right ) \dif {\vec \xi} \nonumber  \\
=&   ({\mathbf R})_{ku} ({\mathbf R})_{jv}  \sum_{m'=-\ell}^\ell  \overline{\rho_\ell^{m',m } } \sum_{n'=-t}^t \overline{\rho_t^{n',n}}\nonumber \\
& \left (
  \frac{\im \nu \alpha^{3+\ell +t}(-1)^{\ell }}{ 2  (\ell +1) (t+2) (2\ell +1)(2t+1)} {\vec e}_u \cdot  
\int_{B}  {\vec \xi} \times \left (  \overline{ H_\ell^{m'} ( {\vec \xi}  )} (  {\vec F}_{B,{\vec e}_v}^{(0),t,n'}     + {\vec F}_{B,{\vec e}_v }^{(1),t,n'} ) \right ) \dif {\vec \xi}  \right . \nonumber \\
 & +  \left ( 1 - \mu_r^{-1} \right ) \frac{ \alpha^{3+\ell +t} (-1)^{\ell }}{(2\ell+1)(2t+1)} {\vec  e}_u \cdot  \int_{B}  \overline{H_\ell^{m'} ({\vec \xi} ) } \left ( \frac{1}{t+2} \nabla  \times {\vec F}_{B,{\vec e}_v }^{(1),t,n'}    \right ) \dif {\vec \xi}  \nonumber \\ 
 &   +\left .  \left ( 1 - \mu_r^{-1}  \right )  \frac{\alpha^{3+\ell +t} (-1)^{\ell }}{(2\ell+1)(2t+1)} {\vec  e}_u \cdot  \int_{ B } \overline{ H_\ell^{m'} (  {\vec \xi} )}   \left ( \frac{1}{t+2} \nabla_{\xi} \times {\vec F}_{ B ,{\vec e}_v}^{(0),t,n'}     \right ) \dif {\vec \xi} \right ) \nonumber  ,
 \end{align}
which follows by applying similar arguments to those in the proof of Proposition 5.3 in~\cite{Ammari2015} and Theorem 3.1 in~\cite{LedgerLionheart2015}. So that
\begin{align} 
 ({\mathbf C}_{kj}^{{\rm C}, \ell t}  [{\mathbf R} (B) ] )_{mn} =  & ({\mathbf R})_{ku} ({\mathbf R})_{jv}  \sum_{m'=-\ell}^\ell  \overline{\rho_\ell^{m',m } } \sum_{n'=-t}^t \overline{\rho_t^{n',n}}
( {\mathbf C}_{uv}^{{\rm C}, \ell t} [B])_{m' n'}  \nonumber \\
= &   ({\mathbf R})_{ku} ({\mathbf R})_{jv}  \overline{{\mathbf q}_{\ell}^m} ( {\mathbf C}_{uv}^{{\rm C}, \ell t} [B])   (  \overline{{\mathbf q}_{r }^n})^T , \nonumber \\
 ({\mathbf N}_{kj}^{{\rm C}, \ell t}  [{\mathbf R} (B) ] )_{mn} =  & ({\mathbf R})_{ku} ({\mathbf R})_{jv}  \sum_{m'=-\ell}^\ell  \overline{\rho_\ell^{m',m } } \sum_{n'=-t}^t \overline{\rho_t^{n',n}}
( {\mathbf N}_{uv}^{{\rm C}, \ell t} [B])_{m' n'}  \nonumber \\
= &   ({\mathbf R})_{ku} ({\mathbf R})_{jv}  \overline{{\mathbf q}_{\ell}^m} ( {\mathbf N}_{uv}^{{\rm C}, \ell t} [B])   (  \overline{{\mathbf q}_{r }^n})^T , \nonumber 
\end{align}
%
%
where 
\begin{align}
{\mathbf q}_\ell^m := (\rho_\ell^{-\ell,m}, \ldots, \rho_\ell^{\ell,m}). \nonumber
\end{align}
Introducing ${\mathbf Q}_\ell ({\mathbf R})$ from (\ref{eqn:defineq}) completes the proof.
\end{proof}

\subsection{HGMPT coefficients  invariant under the action of a symmetry group}
In this section, we consider how the voltage in a source--receiver pair changes if 1) the coils rotate and the object is fixed; 2) if object rotates and the coils are fixed and then we also relate the two situations. Next, we consider a scalar EIT problem where a procedure has already been established for determining HGPTs coefficients invariant under the action of a symmetry group before presenting an approach to determine the HGMPTs coefficients that are invariant under the action of a symmetry group in the vectorial eddy current case. 

\subsubsection{Changes in voltage due to object rotation}
If a coil arrangement rotates, with the rotation described  by ${\mathbf R}$, so that a  new transmit location is ${\vec x}^{\rm{s}'}= {\mathbf R} {\vec x}^{\rm{s}}$ and its dipole moment is
${\vec m}' = {\mathbf R} {\vec m}$, the background magnetic field at the origin due a transmitter at ${\vec x}^{\rm{r'}}$ can be expressed in terms of the field obtained from a transmitter at ${\vec x}^{\rm{s}}$ as
\begin{align*}
({\vec H}_0 ' ({\vec 0} ))_i &=  ({\vec D}_x^2  G( {\vec x}, {\vec 0}) |_{{\vec x}= {\vec x}^{\rm{s}'}}  )_{ij}   ({\vec m}' )_j  \nonumber \\
& = ({\mathbf R})_{ip} ({\mathbf R})_{jq} ({\vec D}_x^2 G({\vec x} , {\vec 0}) |_{{\vec x}= {\vec x}^{\rm{s}} } )_{pq}  ({\mathbf R})_{jt}  ({\vec m} )_t \nonumber \\
& = ({\mathbf R})_{ip}({\vec D}_x^2 G({\vec x} , {\vec 0}) )|_{{\vec x}=  {\vec x}^{\rm{s}}} )_{pq}   ({\vec m} )_q =  ({\mathbf R})_{ip} ({\vec H}_0 ({\vec 0}) )_p ,
\end{align*}
which follow using (\ref{eqn:h0dipole}) and the properties of ${\mathbf R}$.

We can also predict how the the coefficients of HGMPTs that we have derived in Theorem~\ref{thm:harmgmpt} will transform if the coils are fixed and the object rotates. The coefficients of 
HGMPT
are defined by two sets of indices; a set of tensorial indices, which we denote by subscripts $k,j$,  and a further set of indices, denoted by the subscripts $\ell,p,t,q$. The rank of the HGMPT is $2 +\ell +t$ and the indices $k,j,p,q$ are used to identify different tensors of this rank, specifically  $1\le k,j\le3$, $0\le \ell \le M-1$, $0\le t\le M-1-\ell$, $-\ell \le p \le \ell$ and $-t \le q \le t$. In the simplest case, where  $M=1$, we have ${\mathfrak M}_{kj}^{{\rm H}, \ell p r q }\equiv  {\mathfrak M}_{kj}^{{\rm C},\ell m r n } ={\mathfrak M}_{kj}^{00}= {\mathcal M}_{kj}$ so that HGMPTs and GMPTs in this case agree with the rank 2 MPT coefficients (up to a scaling dependent on the definition of  $I_0^0({\vec \xi})$). Recall, that in~\cite{ledger2022} we chose the harmonic polynomials to be defined so that  $\langle I_n^m({\vec x}), I_n^k({\vec x}) \rangle_{L^2(\partial S)}=\delta_{mk}$ where $\langle u,v\rangle_{L^2(\partial S)} : = \int_{\partial S} u \overline{v} \dif {\vec x}$ denotes the $L^2$ inner product over the surface of the unit sphere $S$. If an object is rotated as $B'={\mathbf R}(B)$, the rank 2 MPT coefficients  of the transformed object in terms of those for the original configuration are
\begin{align*}
{\mathcal M}_{ij} '= ({\mathcal M}[{\mathbf R}(B) ])_{ij}=({\mathbf R})_{ip} ({\mathbf R})_{jq} ( {\mathcal M} [ B])_{pq}=({\mathbf R})_{ip} ({\mathbf R})_{jq}  {\mathcal M}_{pq},
\end{align*}
with Section~\ref{sect:rot} providing the extension for HGMPTs.

We now consider how the voltage changes if an object $B$ is fixed in position and both transmit and receive coils simultaneously rotate by the same rotation matrix ${\mathbf R}$. In this case, the voltage induced in a source--receiver pair $(\rm{s},\rm{r}) $ with dipole moments ${\vec f}$ and ${\vec d}$ and the prime indicates the rotated quantities is
\begin{align}
V_{{\rm{s} \rm{r}}}'=& 
 \sum_{\ell =0}^{M-1} \sum_{t =0}^{M-1-\ell}  
   \sum_{p=-\ell}^\ell  \sum_{q=-t}^t   
{ {f}_i}'  \left ({\vec D}_{ x}^{2}\left (    K_{\ell}^p ({\vec x} )   \right )|_{{\vec x}=  {\vec x}^{{\rm{r}}'}}  \right )_{ik}  
(  { \mathfrak M} [B]) _{kj}^{{\rm C}, \ell p t q }   \left ( {\vec D}_x^2 \left (   { K_{t}^q ({\vec x})} \right ) |_{{\vec x}= {\vec x}^{{\rm{s}'} }} \right )_{jo}
  {d_o}' \nonumber .
  \end{align}
  Then,  noting that $ K_{\ell}^p ({\vec x}^{{\rm{r}}'})$ transforms in a similar way to $H_{\ell}^p ({\vec x}^{{\rm{r}}'}) =H_{\ell}^p ( {\mathbf R} {\vec x}^{{\rm{r}}}) $, as described in (\ref{eqn:rotsphm}), we get
\begin{align}
V_{{\rm{s} \rm{r}}}' =&   \sum_{\ell =0}^{M-1} \sum_{t =0}^{M-1-\ell}  
   \sum_{p=-\ell}^\ell  \sum_{q=-t}^t   
({\mathbf R})_{iv} { {f}_v} ({\mathbf R})_{iw}  ({\mathbf R})_{kn}   \sum_{p'=-\ell}^\ell { \rho_\ell^{p',p} (\alpha, \beta, \gamma) }
 \left ({\vec D}_{ x}^{2}\left (  { K_{\ell}^{p'} ({\vec x})} \right ) |_{{\vec x}= {\vec x}^{{\rm{r}} }} \right )_{wn}  
(  {\mathfrak  M} [B]) _{kj}^{{\rm C},\ell p t q } \nonumber \\
&({\mathbf R})_{ow}  ({\mathbf R})_{j m}  ({\mathbf R})_{ou} \sum_{q'=-t}^t { \rho_t^{q',q} (\alpha, \beta, \gamma) }
  \left ( {\vec D}_x^2 \left ( { K_{t}^{q'} ({\vec x})}  \right )|_{{\vec x}=  {\vec x}^{{\rm{s}} }} \right )_{mu}
  {d_w}  \nonumber \\
  =&   \sum_{\ell =0}^{M-1} \sum_{t =0}^{M-1-\ell}  
   \sum_{p=-\ell}^\ell  \sum_{q=-t}^t   
 { {f}_v}   ({\mathbf R})_{kn}   \sum_{p'=-\ell}^\ell   \rho_\ell^{p',p} (\alpha, \beta, \gamma) 
 \left ({\vec D}_{ x}^{2}\left (  K_{\ell}^{p'} ({\vec x}) \right ) |_{{\vec x}= {\vec x}^{{\rm{r}} }}  \right )_{vn}  
 ( {\mathfrak  M} [B]) _{kj}^{{\rm C}, \ell p r q } \nonumber \\
&  ({\mathbf R})_{j m}   \sum_{q'=-t}^t { \rho_t^{q',q} (\alpha, \beta, \gamma) }
  \left ( {\vec D}_x^2 \left (  { K_{t}^{q'} ({\vec x}) }\right ) |_{{\vec x} = {\vec x}^{{\rm{s}} }} \right )_{mw}
  {d_w}  \nonumber ,
    \end{align}
by using properties of orthogonal matrices.
Swapping the order of summation gives
\begin{align}
V_{{\rm{s} \rm{r}}}'=&   \sum_{\ell =0}^{M-1} \sum_{t =0}^{M-1-\ell}  
   \sum_{p=-\ell}^\ell  \sum_{q=-t}^t   \sum_{p'=-\ell}^\ell \sum_{q'=-t}^t
    \left ({\vec D}_{ x}^{2}\left (  K_{\ell}^{p} ({\vec x}) \right )|_{{\vec x} ={\vec x}^{{\rm{r}} }} \right )_{vn} 
 { {f}_v}   ({\mathbf R})_{kn}     \rho_\ell^{p,p'} (\alpha, \beta, \gamma) 
 ( {\mathfrak  M}[B]) _{kj}^{{\rm C}, \ell p' t q' } \nonumber \\
&  ({\mathbf R})_{j m}    { \rho_t^{q,q'} (\alpha, \beta, \gamma) }
  \left ( {\vec D}_x^2 \left (  { K_{t}^{q } ({\vec x}) }\right ) |_{{\vec x} ={\vec x}^{{\rm{s}} }}\right )_{mw}
  {d_w}  \nonumber \\
  =&   \sum_{\ell =0}^{M-1} \sum_{t =0}^{M-1-\ell}  
   \sum_{p=-\ell}^\ell  \sum_{q=-t}^t  
    \left ({\vec D}_{ x}^{2}\left (  K_{\ell}^{p} ({\vec x} ) \right ) |_{{\vec x} = {\vec x}^{{\rm{r}} }} \right )_{vn} 
 { {f}_v}  
(  {\mathfrak  M} [{\mathbf R}^T (B) ] )_{nm}^{{\rm C},\ell p t q }  \left ( {\vec D}_x^2 \left (  { K_{r}^{q } ({\vec x} ) }\right ) |_{{\vec x} ={\vec x}^{{\rm{s}} }} \right )_{mw}
  {d_w}  \nonumber ,
    \end{align}
    so that $V_{{\rm{s} \rm{r}}}'$ can also expressed in terms of fixed pair of source and receiver coils and a rotation of  the object by ${\mathbf R}^T$.

Analogously, we have 
\begin{align*}
V_{{\rm{s} \rm{r}}}'=& 
 \sum_{\ell =0}^{M-1} \sum_{r =0}^{M-1-\ell}  
   \sum_{p=-\ell}^\ell  \sum_{q=-t}^t   
{ {f}_i}'  \left ( \left . {\vec D}_{ x}^{2}\left (   \frac{1}{|{\vec x}|^{2\ell +1}}   I_{\ell}^p ({\vec x}) \right ) \right |_{{\vec x} = {\vec x}^{{\rm{r}}'}}\right )_{ik}  
 ( { \mathfrak M}[B]) _{kj}^{{\rm H}, \ell p t q }   \left ( \left . {\vec D}_x^2 \left (   \frac{1}{|{\vec x} |^{2t +1}} { I_{t}^q ({\vec x})} \right ) \right |_{{\vec x} = {\vec x}^{{\rm{s}}'}} \right )_{jo}
  {d_o}' \nonumber \\
=
& \sum_{\ell =0}^{M-1} \sum_{t =0}^{M-1-\ell}  
   \sum_{p=-\ell}^\ell  \sum_{q=-t}^t   
{ {f}_i}  \left ( \left . {\vec D}_{ x}^{2}\left (   \frac{1}{|{\vec x} |^{2\ell +1}}   I_{\ell}^p ({\vec x}) \right ) \right  |_{{\vec x} = {\vec x}^{{\rm{r}} }} \right )_{ik}  
( {\mathfrak   M} [ {\mathbf R}^T (B) ] )_{kj}^{{\rm H}, \ell p t q }  
\left ( \left .  {\vec D}_x^2 \left (   \frac{1}{|{\vec x}  |^{2t +1}} { I_{t}^q ({\vec x} )} \right ) \right |_{{\vec x} = {\vec x}^{{\rm{s}}} } \right )_{jo}
  {d_o} , 
\end{align*}
in terms of the HGMPT coefficients.

\subsubsection{Scalar problem anology}
For a related scalar EIT problem, where object size is not considered, the induced voltage from a source, receiver pair due to the presence of an object $B$ with contrast $k$ can be described as
\begin{equation}
V_{\text{sr}} =  \sum_{\alpha,\beta, |\alpha|=|\beta|=1}^\infty \frac{(-1)^{|\alpha| +|\beta|  }} {\alpha! \beta !} ( \partial_{ x}^\alpha (G({\vec x},{\vec 0}))|_{{\vec x}={\vec x}^{\text{r}}}) M_{\alpha \beta } (\partial_{ x}^\beta ( G({\vec x},{\vec 0}))|_{{\vec x}={\vec x}^{\text{s}}}), \label{eqn:indvolt}
\end{equation}
where $M_{\alpha,\beta}$ denote the coefficients of GPTs in terms of multi-indices, which we show in~\cite{LedgerLionheart2015}, can be expressed in the alternative form
\begin{align*}
V_{\text{sr}} = &  \sum_{p,q=1}^\infty   \frac{1}{{|{\vec x}^{\text{r}}|^{2p+1}} {|{\vec x}^{\text{s}}|^{2q+1}} }   \sum_{i=-p}^{ p }  \sum_{j=-q}^{q }
 {I_{p}^i({\vec x}^{\text{r}})}  M_{qjpi}^{\text{H}}
  {I_{q}^j ({\vec x}^{\text{s}})}  ,  
   \end{align*}
  where  $M_{qjpi}^{\text{H}}$ are what we call the coefficients of harmonic GPTs or HGPTs.
  In~\cite{LedgerLionheart2015}, we describe an approach for reducing the number of independent coefficients of a rank 2 symmetric polarizability tensor using the rotational and reflectional symmetries of an object. This means that in practice for many objects the number of independent objects is much smaller than $6$. Then, in~\cite{ledger2022}, based  on the induced voltage in a source, receiver pair for a related expansion for a scalar EIT type problem involving HGPTs, we developed an approach for determining the symmetric products 
 of harmonic polynomials $I({\vec x})$ and $J({\vec x})$, of possibly different degrees, in the form
 \begin{equation}
 S({\vec x},{\vec y}) = S({\vec y},{\vec x}) = I({\vec x}) J({\vec y}) + J({\vec x}) I({\vec y}) \nonumber , 
 \end{equation}
 that have the property that
 \begin{equation}
  S( {\mathbf R} {\vec x}, {\mathbf R} {\vec y})  =  S(  {\vec x},{\vec y})  , \nonumber
 \end{equation}
 for all matrix representations ${\mathbf R}$ that make up the group ${\mathfrak G}$. This was then applied to reduce the number of independent coefficients of HGPTs associated with objects that are members of a given symmetry group.

 In order to establish the connection with HGMPTs it is useful to rewrite (\ref{eqn:indvolt}) in the alternative form
\begin{equation}
V_{\text{sr}} =  \sum_{\alpha,\beta, |\alpha|=|\beta|=0}^\infty \frac{(-1)^{|\alpha| +|\beta|  }} {(|\alpha|+1)(|\beta|+1) \alpha! \beta !} ( \partial_{ x}^\alpha (\nabla_x (G({\vec x},{\vec 0})))_i |_{{\vec x}={\vec x}^{\text{r}}}) M_{ij}^{\alpha \beta } (\partial_{ x}^\beta (\nabla_x( G({\vec x},{\vec 0})))_j|_{{\vec x}={\vec x}^{\text{s}}}), \label{eqn:expderiv}
\end{equation}
where summation over the tensorial indices $i,j=1,2,3$ is implied and
\begin{equation}
M_{ij}^{\alpha  \beta } : = \int_{\partial B} y_i^\alpha  \phi_{j,\beta} ({\vec y}) \dif {\vec y} , \qquad \phi_{j,\beta} ({\vec y}) := ( \lambda I - K_B^*)^{-1} ( {\vec \nu}_{ x} \cdot \nabla (x_i ^\beta )|_{{\vec x}= {\vec y}} ),\qquad  {\vec y} \in \partial B , \label{eqn:gpt}
\end{equation}
In the above, we have used the notation $y_i^\alpha = y_i {\vec y}^\alpha $. Still further, an alternative form $M_{ij}^{\alpha  \beta } $ can be established in terms of the solution $\psi_{j,\beta}$ to the scalar transmission problem
\begin{subequations}\label{eqn:scalartrans}
\begin{align}
\nabla^2 \psi_{j,\beta}  & = 0 && \text{in $B\cup B^c$} ,\\
[\psi_{j,\beta}]_\Gamma  & = 0 &&  \text{on $\Gamma$} ,\\
\left . \frac{\partial \psi_{j,\beta}}{\partial {\vec n}} \right |_+ - \left .  k \frac{\partial \psi_{j,\beta}}{\partial {\vec n}} \right |_- & = {\vec n} \cdot ({\vec x}^\beta{\vec e}_j) && \text{on $\Gamma$}  ,\\
\psi_{j,\beta}  & \to 0 && \text{as $|{\vec x}| \to \infty$},
\end{align}
\end{subequations}
in the form
\begin{align}
M_{ij}^{\alpha  \beta  }  = (k-1) \int_B  ({\vec x}^\alpha{\vec e}_i)  \cdot ({\vec x}^\beta  {\vec e}_j) \dif {\vec x} + ( k-1)^2 \int_B \nabla {\vec x}^\beta {\vec e}_j  \cdot \nabla \psi_i^\alpha \dif {\vec x} .\label{eqn:scalarhtensym}
\end{align}
In (\ref{eqn:scalartrans}), we note that of the possible combinations of $\nabla_x x_j^\beta$ we need only consider those functions that are harmonic and, since the gradient of a harmonic function is still harmonic, we can restrict ourselves to  ${\vec x}^{\beta}{\vec e}_j$ with $\beta$ being such that the multi-indices lead to polynomials ${\vec x}^{\beta}$ that are harmonic. The result in (\ref{eqn:scalarhtensym}) follows from Lemma 4.3 in~\cite{ammarikangbook} since again only those multi-indices $\alpha$ for which ${\vec x}^\alpha{\vec e}_i $ is harmonic need be considered.  
Furthermore,
\begin{align*}
M_{ij}^{\alpha  \beta  }  = (k-1) \int_B ({\vec x}^\alpha{\vec e}_i)  \cdot ({\vec x}^\beta  {\vec e}_j)   \dif {\vec x} - ( k-1)^2 \int_{B^c} \nabla \psi_j^\beta \cdot \nabla \psi_i^\alpha \dif {\vec x}  -
( k-1)^2 \int_{B}k \nabla \psi_j^\beta \cdot \nabla \psi_i^\alpha \dif {\vec x} ,
\end{align*}
 which is obtained by integration by parts. 

 By beginning from (\ref{eqn:expderiv}), and repeating similar steps to~\cite{LedgerLionheart2015}, we arrive at
\begin{align}
V_{\text{sr}} 
= &  \sum_{\ell ,t=0}^\infty      \sum_{p=-\ell}^{ \ell }  \sum_{q=-t}^{t }
\left ( \left . \nabla_x \left (
\frac{1}{{|{\vec x} |^{2\ell+1}}} I_{\ell}^p ({\vec x}) \right ) \right |_{{\vec x}=
 {\vec x}^{\text{r}}} \right )_k
  M_{kj}^{H,\ell p tq}
\left ( \left . \nabla_x \left (
\frac{1}{{|{\vec x}|^{2t+1}}} I_{t}^q({\vec x}) \right ) \right |_{{\vec x}=
 {\vec x}^{\text{s}}} \right )_j \label{eqn:indvoltharmeitalt},
   \end{align} 
where 
\begin{align}
M_{kj}^{H,\ell pt q}
  =  & (k-1) \int_B   ( I_{\ell}^p({\vec x}) {\vec e}_k )  \cdot
 ( I_{t}^q ({\vec x}) {\vec e}_j )
\dif {\vec x} \nonumber \\
& - ( k-1)^2 \int_{B^c} \nabla \phi_{k,\ell p}  \cdot \nabla \phi_{j, tq} \dif {\vec x}  -
( k-1)^2 \int_{B}k \nabla \phi_{k, \ell p} \cdot \nabla \phi_{j, t q  } \dif {\vec x} ,\label{eqn:indvoltharmeitaltsym},
\end{align}
and
\begin{subequations}
\begin{align}
\nabla^2 \phi_{k,\ell p }  & = 0 && \text{in $B\cup B^c$}, \\
[\phi_{k,\ell p}  ]_\Gamma & = 0 &&  \text{on $\Gamma$} ,\\
\left . \frac{\partial \phi_{k,\ell p }}{\partial {\vec n}} \right |_+ - \left .  k \frac{\partial \phi_{{k,\ell p }}}{\partial {\vec n}} \right |_- & = {\vec n} \cdot  ( I_{\ell}^ p  ({\vec x} ) {\vec e}_k )  && \text{on $\Gamma$} , \\
\phi_{k,\ell p }  & \to 0 && \text{as $|{\vec x}| \to \infty$},
\end{align}
\end{subequations}
 We observe that (\ref{eqn:indvoltharmeitalt}) has a similar form to (\ref{eqn:indvoltharm}) with summation over a set of tensorial indices $1\le k,j\le 3$ and additional summation over $\ell,t =0,1,\ldots$ with $-\ell \le p \le \ell$ as well as $-t \le q \le t$. 
 Still further, for $\mu_r =1 $, then ${\vec \psi}_k^{(0),\ell,u} =  I_\ell^u({\vec \xi}) {\vec e}_k \times {\vec \xi}$ and we can write the coefficients of the HGMPTs in the alternative symmetric form
 \begin{align}
 {\mathfrak M}_{kj}^{{\rm H}, \ell p t q } =&
   \frac{\im \nu \alpha^{3+\ell +t}(-1)^{\ell }}{  2(\ell +1) (t+2) (2\ell+1)(2t+1)} {\vec e}_k \cdot  
\int_B {\vec \xi} \times \left (  I_\ell^p ({\vec \xi})  ({\vec \psi}_j^{(0),t,q} + {\vec \psi}_j^{(1),t,q } )\right ) \dif {\vec \xi}   \nonumber \\
= &  \frac{ \alpha^{3+\ell +t}(-1)^{\ell }}{  2(\ell +1) (t+2) (2\ell+1)(2t+1)} \left (
   \int_B \frac{1}{\im \nu} \nabla \times \mu_r^{-1} \nabla \times {\vec \psi}_j^{(1),t,q} \cdot
   \nabla \times \mu_r^{-1} \nabla \times {\vec \psi}_k^{(1),\ell,p} \dif {\vec \xi} \right . \nonumber \\
&  \left .   - \int_{B \cup B^c } \tilde{\mu}_r^{-1}  \nabla \times  {\vec \psi}_j^{(1),t,q}  \cdot \nabla \times  {\vec \psi}_k^{(1),\ell,p} \dif {\vec \xi}   
   \right ), \label{eqn:harmgmptsym}
\end{align}
by applying similar arguments to Lemma~\ref{lemma:symmetry}. Also, defining $\tilde{\mathfrak M}_{kj}^{{\rm H}, \ell p t q } := 2(\ell +1) (t+2) {\mathfrak M}_{kj}^{{\rm H}, \ell p t q } $, then we see we have the symmetry
$\tilde{\mathfrak M}_{kj}^{{\rm H}, \ell p t q } = \tilde{\mathfrak M}_{jk}^{{\rm H},  t q \ell p}$.

\subsubsection{Procedure to determine invariant HGMPT coefficients}

Given the similarity between  (\ref{eqn:indvoltharmeitalt})  to (\ref{eqn:indvoltharm}) we can proceed as follows to determine the coefficents of HGMPTs and HGPTs (when expressed in the alternative form  (\ref{eqn:indvoltharmeitaltsym})) that are invariant under the action of a symmetry group ${\mathcal G}$:
\begin{enumerate}

\item For  $\ell= t =0$, when the HGPTs and HGMPTs reduce to (complex) symmetric rank 2 tensors, we apply the previous procedure from~\cite{LedgerLionheart2015} to determine the independent coefficients associated with indices $1\le k,j\le 3$.

\item For other cases, and once independent H(G)MPTs coefficients for indices $k$ and $j$ have been identified as above, we propose to use the previously described approach in~\cite{ledger2022} to determine the symmetric products of harmonic polynomials that are invariant under the action of a symmetry group. This, in turn, allows us to additionally identify the independent coefficients for indices $\ell, p,t,q$ for HGMPTs.

\item Once the independent HGMPT coefficients have been identified, we propose to use these as features in classification algorithms where objects of the same symmetry group are grouped together to form classes. We plan to investigate this in a future publication.

\end{enumerate} 

\begin{remark}
We envisage that the above procedure could be used to identify unexploded ordnance (UXOs), landmine components, metallic objects of archeological significance as well as for identifying objects for security screening and for other metal detection applications.
Just as with situation in EIT described in~\cite{ledger2022}, in practice the measured $V_{\rm{sr}}$ will contain unavoidable errors and noise that are associated with measurements.  Still further, buried objects (and other objects that we wish to find) will often be dented and deformed and so in practice a hidden object's symmetries may only hold approximately in practice.
\end{remark}
 
 \section{Conclusion}
 
In this work we have derived complete asymptotic expansions of $({\vec H}_\alpha-{\vec H}_0)({\vec x})$ as $\alpha \to 0$ using both tensorial index and multi-index notation, which provide improved object characterisations using higher order GMPTs as a natural extension of the rank 2 MPT description. We provide splittings of the GMPT object characterisations obtained, which make the magnetostatic contribution to $({\vec H}_\alpha-{\vec H}_0)({\vec x})$ explicit. We have derived symmetry properties of GMPTs, which extend those already known for rank 2 MPTs, and have also obtained explicit formulae for the real and imaginary coefficients of GMPTs, again, extending those already known for MPTs. We have derived results that explain the spectral behaviour of GMPT coefficients (ie their behaviour as a function of frequency) and shown that their behaviour is similar to that of the MPT coefficents. We have also introduced the new concept of harmonic GMPTs, which have fewer coefficients than GMPTs of the same order. We have examined their scaling, translation and rotational properties and provided an approach for determining the coefficients of HGMPTs that are invariant under the action of a symmetry group, which could form a basis of object classification for (H)GMPTs.

\section*{Acknowledgements}
Paul D. Ledger gratefully acknowledges the financial support received from EPSRC in the form of grants  EP/V049453/1 and EP/V009028/1. William R. B. Lionheart gratefully acknowledges the financial support received from EPSRC in the form of grants  EP/V049496/1 and EP/V009109/1 and would like to thank the Royal Society for the financial support received from a Royal Society Wolfson Research Merit Award and  a Royal Society  Global Challenges Research Fund grant CH160063.

\bibliographystyle{plainurl}
\bibliography{paperbib}
 
 \end{document}